\documentclass{amsart}
\usepackage{amsmath,amsthm}
\usepackage{amsfonts,amssymb}
\usepackage{accents}
\usepackage{enumerate}
\usepackage{accents,color}
\usepackage{graphicx}
\usepackage{comment}
\usepackage{hyperref}

\newcommand{\lvt}{\left|\kern-1.35pt\left|\kern-1.3pt\left|}
\newcommand{\rvt}{\right|\kern-1.3pt\right|\kern-1.35pt\right|}
\newtheorem{thm}{Theorem}[section]
\newtheorem{cor}[thm]{Corollary}
\newtheorem{lem}[thm]{Lemma}
\newtheorem{prop}[thm]{Proposition}

\newtheorem{defn}[thm]{Definition}
\usepackage{bm}      % Load the bm package

\theoremstyle{remark}
\newtheorem{rem}{Remark}[section]

 \def\la{{\langle}}
 \def\ra{{\rangle}}

 \def\d{\mathrm{d}}
 \def\e{\mathrm{e}}

 \def\sph{{\mathbb{S}^{d-1}}}
 
 \def\sb{{\mathsf b}}
 \def\sc{{\mathsf c}}
 \def\sd{{\mathsf d}}
  \def\sh{{\mathsf h}}

 \def\sT{{\mathsf T}}
 \def\sP{{\mathsf P}}
 \def\sQ{{\mathsf Q}}

 \def\sW{{\mathsf W}}
 \def\sY{{\mathsf Y}}

 \def\a{{\alpha}}
 \def\b{{\beta}}
 \def\g{{\gamma}}
 \def\k{{\kappa}}
 \def\t{{\theta}}
 \def\l{{\lambda}}
 \def\o{{\omega}}
 \def\s{\sigma}
 \def\la{{\langle}}
 \def\ra{{\rangle}}

 \def\bb{{\bm b}}
 \def\cb{{\bm c}}

 \def\jb{{\bm j}}
 \def\hb{{\bm h}}
 \def\kb{{\bm k}}

 \def\ub{{\bm u}}
 \def\vb{{\bm v}}
 \def\xb{{\bm x}}
 \def\yb{{\bm y}}
 
  \def\Bb{{\bm B}}

 \def\Lb{{\bm L}}
 \def\Pb{{\bm P}}
 \def\Qb{{\bm Q}}
 \def\Rb{{\bm R}}

 \def\Wb{{\bm W}}

 \def\CH{{\mathcal H}}

 \def\CV{{\mathcal V}}

 \def\BB{{\mathbb B}}

 \def\NN{{\mathbb N}}

 \def\RR{{\mathbb R}}
 \def\SS{{\mathbb S}}
 
 \def\VV{{\mathbb V}}

      \def\proj{\operatorname{proj}}

\def\lla{\langle{\kern-2.5pt}\langle}      
\def\rra{\rangle{\kern-2.5pt}\rangle}

\def\bg{{\boldsymbol{\gamma}}}
\def\bk{{\boldsymbol{\kappa}}}

\newcommand{\wh}{\widehat}

\def\f{\frac}

\graphicspath{{./}}
\begin{document}
 
\title[Addition formula for orthogonal polynomials]
{Addition formula for two families of orthogonal polynomials in several variables}
%{Orthogonal polynomials which are eigenfunctions of a partial differential operator}
\author{Yuan~Xu}
\address{Department of Mathematics, University of Oregon, Eugene,
OR 97403--1222, USA}
\email{yuan@uoregon.edu}
\thanks{The author was partially supported by Simons Foundation Grant \#849676}
\date{\today}
%\keywords{}

\begin{abstract}
An addition formula is a closed-form formula for reproducing kernels of orthogonal polynomials of 
several variables; the latter are kernels of orthogonal projection operators. We derive addition formulas 
for two families of orthogonal polynomials, which consist of wrapped products of classical Gegenbauer 
polynomials on the unit ball with either classical Jacobi polynomials on the simplex or product Laguerre 
polynomials, each of which is given explicitly as a multilayer of integrals of a classical orthogonal 
polynomial of one variable. These formulas provide a powerful tool for analysis and allow us, in 
particular, to reduce a substantial portion of analysis on the Fourier orthogonal expansions on the 
domains to that of one variable. 
\end{abstract}

\maketitle

\section{Introduction}
\setcounter{equation}{0}
 
Spherical harmonics are homogeneous harmonic polynomials, and they are orthogonal polynomials on
the unit sphere. Let $\CH_n^d$ be the space of spherical harmonics of degree $n$ in $d$ variables 
and let $\{Y_\ell^n: 1 \le \ell \le \dim \CH_n^d\}$ be an orthonormal basis of $\CH_n^d$. One of the 
fundamental properties of spherical harmonics is an elegant formula, called an {\it addition formula}, 
given by
$$
 \sum_{\ell} Y_\ell^n(\xi) Y_\ell^n(\eta) = Z_n^{\frac{d-2}{2}} (\la \xi, \eta \ra), \qquad \xi, \eta \in \sph,
$$
where $Z_n^\l$ is a Gegenbauer polynomial of degree $n$ in one variable. The sum on the left-hand 
side is the reproducing kernel of $\CH_n^d$, which serves as the kernel of the projection operator
$\proj_n: L^2(\sph) \mapsto \CH_n^d$. As such, the addition formula provides a powerful tool for Fourier 
analysis on the sphere, which helps to reduce, in particular, a substantial portion of analysis on the 
sphere to analysis on the Fourier-Gegenbauer series on $[-1,1]$.  

More recently, reproducing kernels of several other families of orthogonal polynomials of several variables 
are found to admit closed-form formulas, which will also be called addition formulas, taking the cue from 
that of spherical harmonics. The first such formula (\cite{X99}) is for classical Gegenbauer polynomials on 
the unit ball $\BB^d$ of $\RR^d$, followed by one for $h$-spherical harmonics on weighted integrals on 
the unit sphere \cite{X97}, which further implies an addition formula for the classical Jacobi polynomials 
on the simplex in $\RR^d$. In each of these cases, 
the closed-form formula of the addition formula is given by a multilayer of integrals of one classical 
orthogonal polynomial of one variable, which reveals a hidden one-dimensional characteristic of the 
reproducing kernels. The addition formula is one of two intrinsic characteristics for the classical orthogonal 
polynomials in several variables. The other one is the spectral operator, a second-order differential 
operator that has orthogonal polynomials as eigenfunctions. Together, the two characteristics play  
essential roles in recent studies of several problems in approximation theory and harmonic analysis, both 
theoretical and computational, based on Fourier orthogonal expansions. In particular, they and other features 
of regular domains, such as the unit sphere, unit ball, and simplex, have been utilized to obtain elegant 
and revealing results that serve as prototypes in higher dimensions (cf. \cite{DaiX, IPX, KPX2, NPW1, 
NPW2, PX, X05, X21} and references therein). 

Systems of orthogonal polynomials that possess both characteristics are special and rare. Motivated by 
the advances described above, orthogonal polynomials on rotational quadratic domains are explored in 
\cite{OX, X20, X21, X21b}, which show, in particular, that two families of orthogonal polynomials on the 
rotational cones in $\RR^{d+1}$, as well as those on the conic surfaces, possess both a spectral operator 
and an addition formula. The newly introduced conic domains prompted us to explore the common features 
of these previous works and led to a framework developed in \cite{X21}, in the setting of space of homogeneous 
type, for analysis based on the Fourier orthogonal expansion. Under the assumption of the existence of 
a spectral operator and highly localized kernels that are established, in all known cases, via the addition 
formulas, the framework establishes many previous results on the unit sphere, ball, and simplex, as 
well as on the finite rotational cones, including a characterization of best approximation by polynomials 
via the $k$-functional defined by the spectral operator, localized tight frames, positive cubature rules,
and various polynomial inequalities.   

The domains, equipped with appropriate weight functions, for which orthogonal polynomials possess the 
two essential properties, the spectral operator and the addition formula, enjoy rich geometric and algebraic 
structures that permit in-depth study of a variety of problems in higher dimensions. They are special, and 
each reveals phenomena of its own kind. The framework in \cite{X21} shows that they could be treated 
systematically. One naturally asks how many such domains there are. For $d =2$, the existence 
of the spectral operator is characterized in \cite{KrSh}, which shows there are primarily, up to affine 
transformations, five families, which are product Hermite and/or Laguerre polynomials, and classical 
orthogonal polynomials on the unit disk and on the triangle. All five have straightforward extensions to 
higher dimensions, but no classification is known in dimension three or higher. Motivated by the structure 
of classical orthogonal polynomials and those on conic domains studied in \cite{X20}, for which an 
orthogonal basis consists of wrapped products, see the definition in the next section, of orthogonal 
polynomials, we examined possible families among a large family of wrapped product orthogonal polynomials 
in \cite{X26}, and identified two new families on quadratic domains, as well as two families on the surfaces 
of the two domains, which include the ones on rotational cones as special cases.  

The spectral operator and the addition formula are separate entities, but intriguingly, they exist, or do not exist, 
simultaneously in just about all cases that are known. Thus, it is natural to ask if the two new families of
orthogonal polynomials in \cite{X26} that possess spectral operators also possess addition formulas. 
The purpose of this paper is to show that they do and establish these addition formulas. The two families 
are wrapped products of classical orthogonal polynomials on the unit ball with either classical Jacobi 
polynomials on the simplex or product Laguerre polynomials. The addition formula is given in terms of 
the Gegenbauer polynomials in the first case and in terms of the Laguerre polynomials in the second case. 

The paper is organized as follows. The next section is preliminary, where we review the basics of orthogonal 
polynomials of several variables, and recall the addition formulas for classical orthogonal polynomials that 
will be needed in the sequel. The addition formula for the wrapped product of the unit ball and the simplex 
is established in the third section, followed by the one for the wrapped product of the unit sphere and the simplex 
in the fourth section. The case for the wrapped product of the orthogonal polynomials on the unit ball and
the Laguerre polynomials is established in the fifth section, followed by the one on the corresponding 
surface in the sixth section. 

\section{Preliminary}
\setcounter{equation}{0}

The basics of multivariate orthogonal polynomials are recalled in the first subsection, followed by short subsections 
about orthogonal polynomials on the sphere, the unit ball, the simplex, and product Laguerre polynomials, which 
will serve as building blocks in our study of new orthogonal polynomials. The last subsection reviews the definition
of wrapped product orthogonal polynomials. 

\subsection{Orthogonal polynomials of several variables} 
Let $\Omega$ be a set in $\RR^d$, either a domain with non-empty interior or an algebraic surface. Let $W$ be a non-negative
weight function defined on $\Omega$, normalized to have unit integral on $\Omega$, so that the bilinear form 
$$
    \la f,g\ra_W = \int_{\Omega} f(\bm x) g(\bm x) W(\bm x) \d \bm x
$$
is a well-defined inner product on the space of polynomials restricted to $\Omega$, and orthogonal polynomials under this inner
product are well defined. Let $\CV_n^d(\Omega, W)$ be the space of orthogonal polynomials of degree $n$ under this inner product. 
If $\Omega$ has non-empty interior, then
\begin{equation}\label{eq:dimVn}
  \dim \CV_n^d(\Omega, W) = \binom{n+d-1}{n}, \qquad n = 0,1,2, \ldots. 
\end{equation}
If $\Omega$ is a quadratic surface, such as the unit sphere $\sph$ of $\RR^d$, we denote the domain by $\Omega_0$ instead of $\Omega$, 
then the orthogonal polynomials in $\CV_n^d(\Omega_0, W)$ consist of polynomials restricted to $\Omega_0$, so that 
\begin{equation}\label{eq:dimV0n}
    \dim \CV_n^d(\Omega_0,W) =  \binom{n+d-2}{n}+ \binom{n+d-3}{n-1}. 
\end{equation}

The reproducing kernel of $\CV_n^d(\Omega,W)$, denoted by $P_n(W; \cdot,\cdot)$, is uniquely determined by the property
$$
    \int_\Omega f(\bm y) P_n(W; \bm x, \bm y) W(\bm y)\d \bm y = f(\bm x), \qquad \forall f\in \CV_n^d(\Omega, W).
$$
If $\{P_{\nu,n}: 1 \le \nu \le \dim \CV_n^d(\Omega,W)\}$ is an orthogonal basis of $\CV_n^d(\Omega,W)$, then 
\begin{equation}\label{eq:reprod-kernel}
P_n(W; \bm x, \bm y) = \sum_{\nu=1}^{\dim\CV_n^d(\Omega,W)} \frac{P_{\nu,n}(\bm x) P_{\nu,n}(\bm y)}{\la P_{\nu,n},P_{\nu,n} \ra_W},
\end{equation}
and the kernel plays an important role in the study of the Fourier orthogonal series of $f \in L^2(\Omega,W)$ defined by
\begin{equation}\label{def:Fourier}
  f   = \sum_{n=0}^\infty   \sum_{\nu=1}^{\dim\CV_n^d(\Omega,W)}
   \wh f_{\nu,n} P_{\nu,n}, 
  \qquad \wh f_{\nu,n} = \frac{\la f, P_{\nu,n}\ra_W}{\la P_{\nu,n},P_{\nu,n} \ra_W}. 
\end{equation}
Indeed, let $\proj_n(W):  L^2(\Omega, W)\mapsto \CV_n^d(\Omega, W)$ be the projection operator; then the Fourier 
series can be written as 
$$
   f= \sum_{n=0}^\infty  \proj_n(W; f), \quad \hbox{where} \quad 
     \proj_n f = \sum_{\nu=1}^{\dim\CV_n^d(\Omega,\varpi)} \wh f_{\nu,n} P_{\nu,n},
$$
and the operator $\proj_n(W)$ is an integral operator with $P_n(W; \cdot,\cdot)$ as its kernel, 
\begin{equation}\label{def:projPn}
  \proj_n(W; f, \bm x) = \int_\Omega f(\bm y) P_n(W; \bm x, \bm y) W(\bm y) \d \bm y, \quad f\in L^2(\Omega, W).
\end{equation}

For these definitions and orthogonal polynomials of several variables in general, we refer to \cite{DX}. We are interested 
in the case when $P_n(W; \cdot, \cdot)$ satisfies a closed-form formula, which is what we call an addition formula. 
Such a formula exists only in a few cases but provides a powerful tool when it does. We review what is known for 
such a formula below. 

\subsection{Spherical harmonics}
These are orthogonal polynomials on the unit sphere $\sph$ of $\RR^d$. Let $\Delta$ be the usual Laplace operator of $\RR^d$.
A spherical harmonic $Y$ of degree $n$ is a homogeneous polynomial of degree $n$ in $d$ variables that satisfies $\Delta Y =0$.
Let $\CH_n^d$ denote the space of spherical harmonics of degree $n$. Then its dimension is given 
by \eqref{eq:dimV0n}. For $n \in \NN_0$ let $\{Y_\ell^n: 1 \le \ell \le \dim \CH_n^d\}$ be an orthonormal basis of $\CH_n^d$; 
then 
$$
   \frac{1}{\o_d} \int_\sph Y_\ell^n (\xi) Y_{\ell'}^m (\xi)\d\s(\xi) = \delta_{\ell,\ell'} \delta_{m,n},
$$
where $\o_d= {2 \pi^{\f{d}{2}}}/{\Gamma(\f{d}{2})}$ denotes the surface area of $\sph$. 
%$$
%\o_d = \frac{2 \pi^{\f{d}{2}}}{\Gamma(\f{d}{2})}.
%$$ 
An explicit basis of $\CH_n^d$ can be given in spherical coordinates in terms of the Jacobi polynomials 
(see, for example, \cite[p. 116]{DX}). 

Let $\sP_n(\cdot,\cdot)$ denote the reproducing kernel of the space $\CH_n^d$. It satisfies an elegant closed form formula 
(cf. \cite[(1.2.3) and (1.2.7)]{DaiX}),  
\begin{equation} \label{eq:additionF}
  \sP_n(x,y) = \sum_{\ell =1}^{\dim \CH_n^d} Y_\ell^n (\xi) Y_\ell^n(\eta) = Z_n^{\f{d-2}{2}} (\la \xi, \eta \ra),
     \quad \xi, \eta \in \sph, 
\end{equation}
where $Z_n^\l$ is a polynomial of degree $n$ in one variable and it is given by 
\begin{equation} \label{eq:Zn}
    Z_n^\l(t) = \frac{n+\l}{\l} C_n^\l(t), \qquad \l = \frac{d-2}{2},
\end{equation}
in terms of the Gegenbauer polynomial $C_n^\l$, the latter is the classical orthogonal polynomial with respect to the 
weight function $(1-t^2)^{\l - \f12}$ on $[-1,1]$. As we mentioned in the introduction, the identity \eqref{eq:additionF} 
is known as the addition formula of spherical harmonics, and this is the name that we have adopted for closed-form 
formulas of reproducing kernels in general. 

\subsection{Orthogonal polynomials on the unit ball} Let $\BB^d$ denote the unit ball of $\RR^d$. The classical orthogonal 
polynomials on the unit ball are orthogonal with respect to the inner product 
\begin{equation}\label{eq:B-ipd}
  \la f, g\ra_\mu^\BB = \bm b_\mu^\BB \int_{\BB^d} f(\bm x) g(\bm x) \bm W_\mu^\BB (\bm x) \d  \bm x,
\end{equation}
where the weight function is defined by 
\begin{equation}\label{eq:B-weight}
  \bm W^\BB_\mu(\bm x) = \big(1-\|\bm x\|^2\big)^{\mu-\f12}, \qquad \mu > -\tfrac12,
\end{equation}
and the normalization constant $\bm b_\mu^\BB$ of the weight function is given by 
\begin{equation}\label{eq:bB}
\bm b_\mu^\BB = \frac{\Gamma(\mu+\f{d+1}{2})}{\pi^{\f d 2}\Gamma(\mu+\f12)}.
\end{equation}
%which is the normalization constant so that $\la 1, 1\ra_\mu^\BB =1$. 

Let $\CV_n\left(\BB^d, \bm W_\mu^\BB\right)$ be the space of orthogonal polynomials of degree $n$ with respect to $\bm W^\BB_\mu$. 
An orthogonal basis for this space can be given explicitly in terms of the Gegenbauer polynomials in Cartesian coordinates. For 
$\bm x \in \RR^d$, define $\bm x_0 =0$ and $\bm x_j = (x_1,\ldots, x_j)$ for $1 \le j \le d$ and for $\kb \in \NN_0^d$, define 
$\kb^j = (k_j, \ldots, k_d)$ for $1 \le j \le d$ and $\kb^{d+1} =0$. For $\kb \in \NN_0$ with $|\kb| = n$, let 
\begin{equation}\label{eq:OP_BB}
 \bm B_{\kb, n}^\mu(\bm x) = \prod_{j=1}^d \left(1-\|\bm x_{j-1}\|^2\right)^{\frac{k_j}{2}} C_{k_j}^{ \mu+ |\kb^{j+1}| + \frac{d-j}{2}} \bigg(\frac{\bm x}{\sqrt{1-\| \bm x_{j-1} \|^2}} \bigg). 
\end{equation}
%where $\l_j = \mu+ |\kb^{j+1}| + \frac{d-j}{2}$. 
Then $\{\bm B_{\kb,n}^\mu: |\kb|=n, \, \kb \in \NN+0^d\}$ is an orthogonal basis of 
$\CV_n\left(\BB^d, \bm W_\mu^\BB\right)$; see \cite[p. 143]{DX} for this basis as well as the formula 
for the $L^2$ norm of $\bm B_{\kb, n}^\mu$. There are other explicit orthogonal bases on the unit ball, 
see, for example, \cite[Chapter 4]{DX}. 
  
Let $\bm P_n(\bm W_\mu^\BB; \cdot, \cdot)$ be the reproducing kernel of the space $\CV_n(\BB^d, \bm W_\mu^\BB)$. 
The kernel satisfies an addition formula, which is the first such formula for orthogonal polynomials on domains in
higher dimensions \cite{X99}. For $\bm x, \bm y \in \BB^d$, 
\begin{align} \label{eq:add-ball}
 \bm P_n\left(\bm W_\mu^\BB; \bm x, \bm y\right) = c_{\mu-\f12} \int_{-1}^1 Z_{n}^{\mu+\frac{d-1}{2}} 
   \left(\la \bm x, \bm y \ra + s \sqrt{1-\|\bm x\|^2}  \sqrt{1-\|\bm y\|^2} \right) 
     (1-s^2)^{\mu-1} \d \bm s,   
\end{align}
where $\mu > 0$, $Z_n^\l$ is defined as in \eqref{eq:Zn}, and the constant $c_\mu$ is given by 
$$
  c_\mu:= \left[ \int_{-1}^1(1-s^2)^{\mu-\f12} \d s \right]^{-1} = \frac{\Gamma(\mu+1)}{\pi^{\f12} \Gamma(\mu+\f12)}. 
$$
the formula holds for $\mu =0$ under the limit 
\begin{equation}\label{eq:limitInt}
   \lim_{\mu \to 0+}  c_{\mu-\f12} \int_{-1}^1 f(t) (1-t^2)^{\mu-1} \d t = \frac{f(1) + f(-1)}{2}. 
\end{equation}
 
\subsection{Orthogonal polynomials on the simplex}
Let $\triangle^d$ be the standard simplex in the positive quadrant $\RR_+^d$, defined by 
$$
  \triangle^d = \{\bm y \in \RR^d: y_1 \ge 0, \ldots, y_d \ge 0, |\bm y| <1\}, 
$$
where $|\bm y| = y_1+\cdots + y_d$, a notation that will be used throughout this work. The classical orthogonal
polynomials on the simplex are defined via the inner product 
\begin{equation}\label{eq:T-ipd}
  \la f, g\ra_\bk^\triangle = \bm b_\bk^\triangle \int_{\triangle^d} f(\bm x) g(\bm x) \bm W_\bk^\triangle (\bm x) \d  \bm x,
\end{equation}
where the weight function is the Jacobi weight of several variables defined by 
\begin{equation}\label{eq:W_simplex}
\bm W_\bk^\triangle(\bm y) = \prod_{i=1}^d y_i^{\k_i} (1-|\bm y|)^{\k_{d+1}}, \quad \bm y \in \triangle^{d},
\end{equation}
and the normalization constant $\bm b_\bk^\triangle$ for this weight function is given by
\begin{equation}\label{eq:bT}
  \bm b_\bk^\triangle =  \frac{\Gamma(|\bk| + d+1)}{\prod_{i=1}^{d+1} (\k_i+1)}.
\end{equation}

Let $\CV_n\left(\triangle^d, \bm W_\bk^\triangle\right)$ be the space of orthogonal polynomials of degree $n$ 
with respect to $\bm W^\triangle_\bk$. For $\bm y \in \RR^d$, we again define $\bm y_0 =0$ and 
$\bm y_j = (y_1,\ldots, y_j)$ for $1 \le j \le d$ and for $\bk \in \RR^{d+1}$, define $\bk^j = (\k_j, \ldots, \k_{d+1})$ 
for $1 \le j \le d$, and define likewise $\kb^j$ for $\kb \in \NN_0^d$ but with $\kb^{d+1} =0$. We then define 
\begin{equation}\label{eq:OP_TT}
 \bm T_{\kb, n}^\bk(\bm y) = \prod_{j=1}^d \left( 1-|\bm y_{j-1}| \right)^{k_j} 
     P_{k_j}^{(|\bk^{j+1}| +2 |\kb^{j+1}| + d - j, \k_j)} \bigg(\frac{2 y_j}{1-| \bm y_{j-1}|}-1\bigg)
\end{equation}
in terms of the Jacobi polynomial $P_n^{(\a,\b)}$. Then $\left \{ \bm T_{\kb, n}^\bk: |\kb| =n, \, \kb \in \NN_0^d\right\}$
is an orthogonal basis for $\CV_n^d(\triangle^d, \bm W_\bk^\triangle)$;
see \cite[p. 150]{DX} for this basis as well as the formula for the $L^2$ norm of $\bm T_{\kb, n}^\mu$ in \cite[p. 151]{DX}. 
  
Let $ \bm P_n\left(\bm W_\bk^\triangle; \cdot,\cdot\right)$ be the reproducing kernel of $\CV_n\left(\triangle^d, \bm W_\bk^\triangle\right)$. 
It satisfies an addition formula: for $\bm x, \bm y \in \triangle^d$,
 \begin{align} \label{eq:add-simplex}
 \bm P_n\left(\bm W_\bk^\triangle; \bm x, \bm y\right) = c_\bk \int_{[-1,1]^{d+1}} Z_{2n}^{|\bk|+d}
  \big(\xi_\triangle(\bm x, \bm y, \bm t) \big)
   \prod_{i=1}^{d+1} (1-t_i^2)^{\k_i-\f12} \d \bm t,   
\end{align}
where
$$
\xi_\triangle(\bm x, \bm y, \bm t) = \sqrt{x_1 y_1}\, t_1 + \ldots \sqrt{x_d y_d} \, t_d + \sqrt{1-|\bm x|} \sqrt{1-|\bm y|} \, t_{d+1} 
$$
and the formula holds under the limit \eqref{eq:limitInt} if one or more $\k_i$ equal to $-\f12$. 

\subsection{Product Laguerre polynomials}
The classical Laguerre polynomials, $L_n^\a$, are orthogonal with respect to the weight function 
$x^\a \e^{-x}$ on $\RR_+  = [0, \infty)$, where $\a > - 1$. For $\bk \in \RR^d$ with $k_i > -1$, $1 \le i \le d$,
the product Laguerre weight is defined by
\begin{equation}\label{eq:W_L}
    \Wb_\bk^L(\xb) = \prod_{i=1}^d x_i^{\k_i} \e^{- |\xb|}, \qquad \xb \in \RR_+^d,
\end{equation}
and $\bb_\bk^L$ is the normalization constant defined by $\la 1,1\ra_\bk^L = 1$, which satisfies
$$
   \bb_\bk^L = \frac{1}{ \Gamma(\k_1+1) \cdots \Gamma(\k_d+1)}. 
$$
An orthogonal basis for $\CV_n(\RR_+^d, \Wb_\bk^L)$ is given by the product Laguerre polynomials 
\begin{equation}\label{eq:prod_L}
    \Lb_{\jb}^\bk (\xb) = L_{j_1}^{\k_1}(x_1) \cdots  L_{j_d}^{\k_d}(x_d), \qquad |\jb|= n. 
\end{equation}
 
Let $J_\a$ be the usual Bessel function and let  $j_{\a}(z) := z^{-\a} J_\a(z)$. For $\a > -\f12$, the Laguerre 
polynomials satisfy a product formula, which is classical \cite{Wa}, %; see also \cite[p. 142]{Th}, 
\begin{align} \label{eq:productL1}
 \frac{L_n^\a(t) L_n^\a(s)}{L_n^\a(0)} = 
     & \frac{2^{\a} \Gamma(\a+1)}{\sqrt{2\pi}} \int_0^\pi L_n^{\a} \left(t +s + 2 \sqrt{t s} \cos \t \right)  \\
       &\times \e^{- \sqrt{t s}  \cos \t}  j_{\a - \f{1}{2}} \left(\sqrt{t s} \sin \t\right) (\sin \t)^{2 \a}\d\t.  \notag
\end{align}
Let $\Pb_n\big(\Wb_\bk^L; \cdot,\cdot\big)$ be the reproducing kernel of the space $\CV_n\big(\RR_+^d, \Wb_\bk^L\big)$.
This kernel satisfies an addition formula when $\k_i > -\f12$, $1 \le i \le d$,
\begin{align}\label{eq:addL}
\Pb_n\big(\Wb_\bk^L; \xb,\yb \big) 
 =\,  & \cb_\bk^L \int_{[0,\pi]^d} L_n^{|\bk|+ d-1} \bigg(|\xb| + |\yb| + 2 \sum_{i=1}^d \sqrt{x_i y_i} \cos \t_i \bigg) \\
      & \times \e^{- \sum_{i=1}^d \sqrt{x_i y_i}  \cos \t_i}  \prod_{i=1}^d j_{\k_i - \f{1}{2}} \left(\sqrt{x_i y_i} \sin \t_i \right) 
      \prod_{i=1}^d (\sin \t_i)^{2 \k} \d \bm \t, \notag
\end{align}
 where the constant $\cb_\bk^L$ is given by 
\begin{equation} \label{eq:const_L}
   \cb_\bk^L = \frac{2^{|\bk|}} { (2 \pi)^{\f d 2}}  \prod_{i=1}^d  \Gamma(\k_i+1).
\end{equation}
This formula follows from the product formula \eqref{eq:productL1} and the identity 
$$
    \Pb_n\big(\Wb_\bk^L; \xb, 0 \big) = L_n^{|\kb|+d-1}(|\xb|),
$$
which follows from the generating function of the Laguerre polynomials (cf. proof of Theorem 2.1 of \cite{X00}). 
In the case one or more $\k_i = -\f12$, a more complicated formula can be derived via integration by parts, which
is a sum of two terms in one variable \cite{Ma} and will be far more involved in several variables. We will not discuss
this limiting case in this paper. 

To define wrapped product orthogonal polynomials studied in \cite{X26}, we need a variant of product
Laguerre polynomials on the domain $\RR_\Sigma^d$ defined by
\begin{equation} \label{eq:R-Sig}
  \RR_\Sigma^d := \left \{\yb \in \RR^d: y_1\ge 0, \ldots, y_{d-1} \ge 0, |\yb| \ge 0\right\}.
\end{equation}
Since $y_d$ could be negative as long as $|\yb| \ge 0$, so $\RR_\Sigma^d$ is different from $\RR_+^d$, but 
the mapping $\yb \in \RR_\Sigma^d \mapsto \xb \in \RR_+^d$ defined by 
$$
    y_1 = x_1, \ldots, y_{d-1} = x_{d-1},\, y_d = x_d - |\xb'|, \qquad \xb = (\xb', x_d)
$$
is a bijection and satisfies $|\xb|= |\yb'|+ |\yb|$. The Laguerre weight function $\Wb_\bk^L(\xb)$ becomes, 
under this change of variables, 
\begin{equation} \label{eq:LaguerreSig}
   \Wb_\bk^\Sigma(\yb) = \prod_{i=1}^{d-1} y_i^{\k_i} |\yb|^{\k_d} \e^{-|\yb|-|\yb'|}. 
\end{equation}
In particular, an orthogonal basis of $\CV_n\big(\RR_\Sigma^d, \Wb_\bk^\Sigma\big)$ consist of 
\begin{equation} \label{eq:basis_S}
  \hat \Lb_{\jb,n}^\bk(\yb) := % L_{j_1}^{\k_1}(y_1) \cdots  L_{j_{d-1}}^{\k_{d-1}}(y_{d-1})L_{j_d}^{\k_d}(|\yb|) 
      \Lb_{\jb,n}^\bk (y_1,\ldots, y_{d-1}, |\yb|), \qquad |\jb|= n.
\end{equation}

\subsection{Wrapped product and polynomials on domains} 
%\begin{defn} \label{defn:wrap}
Let $d_1$ and $d_2$ be positive integers and let $\Omega_1^{d_1} \subset \RR^{d_1}$ and 
$\Omega_2^d \subset \RR^{d_2}$ be two domains. Let $\Wb_1$ be a weight function on $\RR^{d_1}$ 
and $\Wb_2$ be the weight functions defined on $\Omega_2^{d_2}$. The wrapped product weight 
function $\Wb_1 \rtimes \Wb_2$ is defined by 
\begin{equation}\label{eq:wrapW}
\Wb^{\Omega_1\rtimes \Omega_2}(\xb,\yb) =\Wb_1\left( \frac{\xb} {\rho(\yb)} \right) \Wb_2(\yb) , 
   \qquad (\xb, \yb) \in \Omega_1^{d_1} \rtimes \Omega_2^{d_2}%\VV^{d_1,d_2},
\end{equation}
where the wrapped domain $\Omega_1^{d_1} \rtimes \Omega_2^{d_2}$ is defined by 
\begin{equation} \label{eq:wrapDomain}
\Omega_1^{d_1}  \rtimes \Omega_2^{d_2}:= \left\{(\xb,\yb): \frac{\xb}{\rho(\yb)} \in \Omega_1^{d_1} , 
\quad \yb \in \Omega_2^{d_2} \right\}, 
\end{equation}
in which $\rho: \Omega_2^{d_2} \mapsto \RR_+$ is nonnegative and specified according to one of the two cases,
\begin{enumerate}[    ]
\item[] {\bf Case 1}. $\rho$ is a polynomial of degree $1$; \smallskip
\item[] {\bf Case 2}. $\rho$ is the square root of a polynomial of degree at most $2$, and $\Omega_1^{d_1} $ 
and $\Wb_1$ are both centrally symmetric, that is, $\xb \in \Omega$ implies $- \xb \in \Omega$ and 
$\Wb(\xb) = \Wb(-\xb)$ for all $\xb \in \Omega$.
\end{enumerate}
 
This is the definition given in \cite{X26}, which is motivated by \cite{K75} for $d =2$, as well as classical families 
of orthogonal polynomials in several variables \cite{DX} and those families on rotationally conic domains 
studied recently in \cite{X20}. Let $\{\Pb_{\kb, m}: |\kb| = m, \kb \in \NN_0^{d_1}\}$ be an orthogonal basis of 
$\CV_n(\Omega_1^{d_1}, \Wb_1)$ and denote the square of its norm by $\hb_{\kb,m}(\Wb_1)$. For each 
positive integer $m$, let $\big\{\Rb^{(m)}_{\jb, n}: |\jb| = n, \jb \in \NN_0^{d_2}\big\}$ be an orthogonal basis for 
$\CV_n\left(\Omega_2^{d_2}, \rho^{2m + d_1} \Wb_2\right)$ and denote the square of its norm by 
$\hb_{\jb,m}(\rho^{2m+d_1}\Wb_2)$. Then the wrapped product orthogonal polynomials are defined in 
\eqref{eq:wrapOP}.

\begin{prop} \label{prop:wrapped_orth}
Let polynomials $\Qb_{\jb, \kb, m}^n$ of $d_1+d_2$ variables be defined by
\begin{equation}\label{eq:wrapOP}
   \Qb_{\jb, \kb, m}^n(\xb, \yb) = \Rb_{\jb,n-m}^{(m)}(\yb) [\rho(\yb)]^m \Pb_{\kb,m}\left(\frac{\xb}{\rho(\yb)}\right),  
\end{equation}
where $0 \le m \le n$, $|\kb| = m$ with $\kb \in \NN_0^{d_1}$, and $|\jb| = n-m$ with $\jb \in \NN_0^{d_2}$. 
Then the collection $\left\{  \Qb_{\jb, \kb, m}^n: |\kb| = m,\,  |\jb| = n-m, 0 \le m \le n\right\}$ is an orthogonal basis 
of the space $\CV_n\big(\Omega_1^{d_1}  \rtimes \Omega_2^{d_2}, \Wb^{\Omega_1\rtimes \Omega_2}\big)$.
\end{prop}

This is Proposition 3.2 in \cite{X26}. The ordinary product is the wrapped product with $\rho(\yb) = 1$. 
By \eqref{eq:OP_BB}, the Gegenbauer polynomials on the unit ball $\BB^d$ satisfy 
$$
\Bb_{\jb, n}^{\mu,d} (\xb,\yb) = \Bb_{\jb, n-m}^{\mu, d_2}(\yb) (1-|\yb|)^m \Bb_{\kb,m}^{\mu,d_1} 
\bigg(\frac{\xb}{\sqrt{1-\|\yb\|^2}}\bigg), 
$$
where $d = d_1+d_2$, which corresponds to $\BB^d = \BB^{d_1} \rtimes \BB^{d_2}$ with $\rho(\yb) = 1-|\yb|$. 
The similar structure holds also for the Jacobi polynomials on the simplex \eqref{eq:OP_TT}. 

The wrapped product orthogonal polynomials are studied in \cite{X26}, where those that are eigenfunctions of
a second-order differential operator are identified, which consists of two new families on 
$\BB^{d_1} \rtimes \triangle^{d_2}$ and $\BB^{d_1} \rtimes \RR_\Sigma^{d_2}$ that have not been studied
before if either $d_1 > 1$ or $d_2 >1$. These are the two families that we shall study in the next section and
Section 5. 

\section{Ball-Simplex polynomials on the domain $\BB^{d_1} \rtimes \triangle^{d_2}$}
\setcounter{equation}{0}

In this section, we consider the first of the two new families of orthogonal polynomials in \cite{X26} that possess 
a spectral operator, which consists of orthogonal polynomials on $\BB^{d_1} \rtimes \triangle^{d_2}$.
These polynomials are reviewed in the first subsection. The addition formula for the reproducing kernels 
is derived in the second subsection. The third section consists of a concise introduction to applications of 
the addition formula to the Fourier orthogonal expansions.
 
\subsection{Orthogonal polynomials}
Let $d_1$ and $d_2$ be two positive integers. Let $\BB^{d_1} \rtimes \triangle^{d_2}$ be the quadratic domain 
in $\RR^{d_1+d_2}$ defined by
$$
 \BB^{d_1}\! \rtimes \triangle^{d_2} = \left\{(\bm x,\bm y) \in \RR^{d_1} \times \RR_+^{d_2}: \|\bm x\| \le 1- |\bm y| \le 1\right\}.
$$
For $\mu > -\f12$ and $\bm\kappa \in \RR^{d_2}$ such that $\k_i > -1$, $1 \le i \le d_2$, we define  
$$
  \bm W_{\mu,\bk}^{\BB\rtimes \triangle}(\bm x,\bm y) = \prod_{i=1}^{d_2} |y_i|^{\k_i} \left[ (1-|\bm y|)^2- \|\bm x\|^2 \right]^{\mu - \f12}  %(1-|\bm y|)^{\k_{d_2+1}}
$$
for $(\xb,\yb) \in \BB^{d_1}\! \rtimes \triangle^{d_2}$ and consider orthogonal polynomials with respect to the 
inner product defined by
\begin{equation} \label{eq:ipdOmega}
  \la f, g\ra_{\mu, \bk}^{\BB\rtimes \triangle}  = \bm b_{\mu,\bk}^{\BB\rtimes \triangle}
     \int_{\BB^{d_1} \! \rtimes \triangle^{d_2}} f(\bm x,\bm y) g(\bm x,\bm y) \bm W_{\mu,\bk}(\bm x,\bm y) \d \bm x \d \bm y, 
\end{equation}
where $ \bm b_{\mu,\bk}$ is the normalization constant so that $ \la 1, 1\ra_{\mu, \bk} =1$ and it satisfies 
\begin{equation} \label{eq:bOmega}
  \bm b_{\mu,\bk}^{\BB\rtimes \triangle} = \bm b^\BB_\mu \times \bm b^\triangle_{\bm \gamma} 
       \quad\hbox{with}\quad \bm \gamma = (\bk, d_1+ 2\mu -1) \in \RR^{d_2+1}, 
\end{equation}
in terms of the normalization constant $\bm b_\mu^\BB$ in \eqref{eq:bB} for $\bm W_\mu^\BB$ and the constant
$\bm b_{\bm \gamma}^\triangle$ in \eqref{eq:bT} of $\bm W_{\bm \gamma}^\triangle$. 

For $d_1, d_2 \ge 1$, orthogonal polynomials on $\BB^{d_1} \rtimes \triangle^{d_2}$ equipped with 
$\bm W_{\mu,\bk}^{\BB\rtimes \triangle}$ are studied recently in \cite{X26}, whereas the case $(d_1,d_2) = (d,1)$
has been studied in \cite{X20}, where the variables were chosen as $(\xb,t)$, for which the domain is a rotating cone 
in $\RR^{d+1}$. Let $\CV_n\big(\BB^{d_1} \rtimes \triangle^{d_2}, \bm W_{\mu,\bk}^{\BB\rtimes \triangle}\big)$ be 
the space of orthogonal polynomials of degree $n$ with respect to the inner product $\la \cdot, \cdot\ra_{\mu, \bk}$. 
Then 
$$
  \dim \CV_n\big(\BB^{d_1} \rtimes \triangle^{d_2}, \bm W_{\mu,\bk}^{\BB\rtimes \triangle}\big)  =\binom{n+d_1+d_2 -1}{n}. 
$$
%Parametrising the domain $\VV^{d_1,d_2}$ by settign $\bm x = (1- | \bm y|) \bm u$, it follows readily that 
%\begin{equation} \label{eq:intOmega}
%  \int_{\VV^{d_1,d_2}} f(\bm x, \bm y) \d \bm x \d \bm y =   \int_{\triangle^{d_2}} (1-|\bm y|)^{d_1} \int_{\BB^{d_1}}
%     f\big( (1-|\bm y|) \bm u, \bm y\big) \d \bm u\, \d \bm y.
%\end{equation}
 Let $\{\bm B_{\bm k, m}^{\mu}: |\bm k|= m, \, \bm k \in \NN_0^{d_1}\}$ be an orthogonal basis for 
$\CV_m(\BB^{d_1},\bm W^\BB_\mu)$ such as the one given in \eqref{eq:OP_BB}, and let 
$\{\bm T_{\bm j, m}^{\bg}: |\bm j| = m, \, \bm j\in \NN_0^{d_2}\}$ be an orthogonal basis for
$\CV_{m}\left(\triangle^{d_2},\bm W^\triangle_{\bg}\right)$, such as the one given in \eqref{eq:OP_TT}. 
We denote by $\bm h_{\bm k,m}^{\mu,\BB}$ and $\bm h_{\bm j,m}^{\bg,\triangle}$ the squared norm
$$
  \bm h_{\bm k,m}^{\mu, \BB} =  \langle \bm B_{\bm k, m}^{\mu}, \bm B_{\bm k, m}^{\mu}\rangle_{\mu}^\BB
 \quad and \quad \bm h_{\bm j,m}^{\bg,\triangle} = \langle \bm T_{\bm j, m}^{\bg}, \bm T_{\bm j, m}^{\bg}  \rangle_{\bg}^\triangle.
$$

\begin{prop}
%Let $\{\bm B_{\bm k, m}^{\mu}: |\bm k|= m, \, \bm k \in \NN_0^{d_1}\}$ be an orthogonal basis for 
%$\CV_m(\BB^{d_1},\bm W_\BB^\mu)$ with norm $\bm H_{\bm k, m}^\mu$ and 
%let $\{\bm T_{\bm j, n-m}^{(\bk, \alpha +2m)}: |\bm j| = n-m, \, \bm j\in \NN_0^{d_2}\}$
%be an orthonormal basis for $\CV_{n-m}\left(\triangle^{d_2},\bm W^\triangle_{(\bk, \alpha+2m)}\right)$, where $\alpha = d_1+ 2 \mu -1$.
For $\mu > - \f12$, $\bk \in \RR^{d_2}$ with $\k_j > -1$, $1 \le j \le d_2$, define 
\begin{equation} \label{eq:Qbasis}
 \bm Q_{\bm j, \bm k, m}^n(\bm x,\bm y) = \bm T_{\bm j, n-m}^{(\bk, \alpha+2m)}(\bm y) (1-|\bm y|)^m 
      \bm B_{\bm k, m}^\mu \left(\frac{\bm x}{1-|\bm y| }\right),
\end{equation}
where $\alpha = 2 \mu + d_1-1$. Then 
$\{\bm Q_{\bm j, \bm k, m}^n: |\bm k| = m, \, |\bm j| = n-m, \, 0 \le m  \le n, \, \bm k \in \NN_0^{d_1}, \, \bm j \in \NN_0^{d_2}\}$ 
is an orthogonal basis of $\CV_n\big(\BB^{d_1} \rtimes \triangle^{d_2}, \bm W_{\mu,\bk}^{\BB\rtimes \triangle}\big)$. 
Moreover, the squared norm of $\bm Q_{\bm j, \bm k, m}^n$ is given by 
\begin{equation} \label{eq:Q_OPNorm}
   \bm h_{\bm j, \bm k, m}^{n; \mu,\bk}  %= \la \bm Q_{\bm j, \bm k, m}^n, \Qb_{\bm j,\bm k, m}^n \ra_{\mu, \bk} 
     %  =   \frac{\bm b^\triangle_{(\bk,\alpha)}}{\bm b^\triangle_{(\bk, \alpha + 2m)}}
       = \frac{(2\mu+d_1)_{2m}} {(|\bk|+ 2\mu + d_1 + d_2)_{2m}}\bm h_{\bm k,m}^{\mu, \BB}  
       \bm h_{\bm j,m}^{(\bk, \a+2m), \triangle}. 
\end{equation}
\end{prop}

This is Proposition 4.1 of \cite{X26} and the case $d_2 = 1$ was studied in \cite{X20}. The case $d_1 =1$ is 
particularly important for our study, and we state it as a corollary with $d_2 = d$, for which the polynomial 
$\bm B_{\bm j, m}^\mu$ degenerates to the Gegenbauer polynomial $C_m^\mu$ in one variable, 
and we denote the squared norm of $C_m^\mu$ by $h_m^\mu$.  

\begin{cor}
For $0 \le m \le n$, $\bm j \in \NN_0^{d}$, $x \in [-1,1]$ and $\bm y \in \RR^d$, define 
\begin{equation} \label{eq:Qbasis_d1=1}
 \bm Q_{\bm j, m}^n(x,\bm y) = \bm T_{\bm j, n-m}^{(\bk, 2\mu+2m)}(\bm y) (1-|\bm y|)^m 
      C_m^\mu \left(\frac{x}{1-|\bm y| }\right). 
\end{equation}
Then the collection $\{\bm Q_{\bm j, m}^n:   |\bm j| = n-m, \, 0 \le m  \le n,   \, \bm j \in \NN_0^{d}\}$ 
is an orthogonal basis of $\CV_n\big(\BB^1 \rtimes \triangle^{d}, \bm W_{\mu,\bk}^{\BB\rtimes \triangle}\big)$. 
Moreover, the squared norm of $\Qb_{\bm j, m}^n$ is given by %, with $\a= 2\mu+d_1-1$, 
\begin{equation} \label{eq:Q_OPNorm_d1=1}
   \bm h_{\bm j, m}^{n; \mu,\bk}  = \la \bm Q_{\bm j, m}^n, \bm Q_{\bm j, m}^n \ra_{\mu, \bk}^{\BB\rtimes \triangle}  
       %=   \frac{\bm b^\triangle_{\bk+2\mu \bm e}}{\bm b^\triangle_{(\bk, 2\mu+ 2m)}}
       = \frac{(2\mu+1)_{2m}} {(|\bk|+ 2\mu + d+1)_{2m}} h_m^\mu \bm h_{\bm j, m}^{(\bk, 2m+2\mu), \triangle}.
\end{equation}
\end{cor}

The domain $\BB^{1} \rtimes \triangle^{d} = [-1,1] \rtimes \triangle^{d}$ is a distorted simplex, not the 
standard one. In the case of $d =1$, $\BB^{1} \rtimes \triangle^1 = [-1,1]\rtimes [0,1]$ is the triangle 
$\{(x,y): |x| \le y \le 1\}$ of $\RR^2$, which can be mapped to the standard triangle 
$\triangle^2 = \{(x_1,x_2): x_1, x_2 \ge 0, x_1+x_2 \le 1\}$ by an affine transformation. The same holds for $d > 1$.
The space $\CV_n\big(\BB^{1} \rtimes \triangle^d, \bm W_{\mu,\bk}^{\BB\rtimes \triangle}\big)$ can be affine 
transformed to the space 
$\CV_n(\triangle^{d+1}, \bm W^\triangle_{\bm \b})$ on the simplex $\triangle^{d+1}$ with
\begin{equation}\label{eq:beta}
  \bm \beta = \left(\k_1, \ldots, \k_{d-1}, \mu - \tfrac12, \mu-\tfrac12, \k_d\right) \in \RR^{d+2}. 
\end{equation}
This provides a second explicit orthogonal basis for $\CV_n\big(\BB^1 \rtimes \triangle^{d}, 
\bm W_{\mu,\bk}^{\BB\rtimes \triangle}\big)$, which plays an important role in deducing an addition formula
for our new orthogonal polynomials with $d_1 > 1$ in the next section. 

\begin{thm} \label{thm:QOPd1=1}
The linear transformation $(x, \bm y) \in [-1,1] \rtimes \triangle^{d} \mapsto \bm u \in \triangle^{d+1}$ defined by
$$
  \bm u = \phi(x, \bm y) = \left(y_1, \ldots, y_{d-1}, \, \tfrac12(1-|\bm y|-x),  \tfrac12(1-|\bm y|+x) \right)
$$
is a bijection between $[-1,1]\rtimes \triangle^{d_2}$ and  $\triangle^{d+1}$. Moreover, under this linear transformation, 
the weight function $\bm W_{\mu,\bk}^{\BB\rtimes \triangle}(x, \bm y)$ with $\bk = (\k_1,\ldots, \k_d)$ becomes
$$
    \bm W_{\bm \b}^\triangle (\ub) = 2^{2\mu-1} u_1^{\k_1} \cdots u_{d-1}^{\k_{d-1}} 
      (1-|\bm u|)^{\k_d} u_d^{\mu -\f12} u_{d+1}^{\mu -\f12}.
$$
In particular, let $\{\bm T_{\bm j, n}^{\bm \b}: |\bm j|  = n, \, \bm j \in \NN_0^{d+1}\}$ be an orthogonal basis
of $\CV_n(\triangle^{d+1}, \bm W^\triangle_{\bm \b})$ on the simplex $\triangle^{d+1}$; then
$$
   \bm Q_{\bm j, n}^{\mu, \bk} (x,\bm y) = \bm T_{\bm j, n}^{\bm \b}\left(y_1, \ldots,  y_{d-1}, 
       \tfrac12(1-|\bm y|-x),  \tfrac12(1-|\bm y|+x) \right)
$$
for $|\bm j| = n$ consists of an orthogonal basis for 
$\CV_n\big([-1,1]\rtimes \triangle^d, \bm W_{\mu, \bk}^{\BB\rtimes \triangle}\big)$. 
\end{thm}

\begin{proof}
The inverse of the linear transform $\ub \mapsto (x, \bm y)$ is easily seen as given by 
$$
  x= u_{d+1}-u_d, \, y_1 = u_1, \ldots, y_{d-1} = u_{d-1}, \, y_d = 1-|\bm u|.
$$
It follows that $u_{d+1}-u_d = x$ and $u_d+u_{d+1} = 1 - |\bm y|$. In particular, $|x| \le 1-|\bm y|$ for 
$(x, \bm y) \in \BB^{1} \! \rtimes \triangle^{d_2}$ 
becomes $|u_d- u_{d+1}| \le u_d+u_{d+1}$, which is equivalent to $u_d \ge 0$ and $u_{d+1} \ge 0$. Moreover, $1-|\bm u| = y_d$, 
so that $(x,\bm y)$ is mapped to $\bm u \in \triangle^{d+1}$. Moreover, the Jacobian $\partial (x, \bm y)/\partial \bm u = 2$, which
shows that the linear transformation is a bijection. Furthermore, 
$$
  \bm W_{\mu,\bm \k}^{[-1,1]\rtimes \triangle}(x, \bm y) = \prod_{i=1}^d y_i^{\k_i} (1-|\bm y|-x)^{\mu-\f12} (1-|\bm y|+x)^{\mu-\f12} 
     = \bm W_{\bm \b}^\triangle (\ub)  
$$
with $\bm u \in \triangle^{d+1}$. Since the affine transformation preserves orthogonality, an orthogonal basis for 
$\CV_n\big(\BB^1 \rtimes \triangle^{d_2}, \bm W_{\mu,\bk}^{\BB\rtimes \triangle}\big)$ can be derived from 
that of $\CV_n(\triangle^{d+1}, \bm W^\triangle_{\bm \b})$.
\end{proof}

\subsection{Addition formula}
Let $\bm P_n^{d_1,d_2}(\bm W_{\mu, \bk}^{\BB \rtimes \triangle}; \cdot,\cdot)$ be the reproducing kernel of the space 
$\CV_n\big(\BB^{d_1} \!\rtimes \triangle^{d_2}, \bm W_{\mu,\bk}^{\BB\rtimes \triangle}\big)$. In terms of the orthogonal 
basis in \eqref{eq:Qbasis}, the kernel is given by
\begin{equation}\label{eq:reprodP}
\bm P_n^{d_1,d_2}\! \left(\bm W_{\mu, \bk}^{\BB\rtimes \triangle}; (\bm x,\bm y), (\bm u, \bm v)\right) = \sum_{m=0}^n 
  \sum_{|\bm j| = n-m} \sum_{|\bm k| = m} \frac{\bm Q_{\bm j, \bm k, m}^n(\bm x,\bm y)\bm Q_{\bm j, \bm k, m}^n(\bm u,\bm v)}{
   \la \bm Q_{\bm j, \bm k, m}^n,\bm Q_{\bm j, \bm k, m}^n\ra_{\mu, \bk}^{\BB\rtimes \triangle} }.
\end{equation}
Our goal is to derive a closed-form formula, called an addition formula, for this kernel. 

We first consider the case of $\BB^1 \rtimes \triangle^{d}$ equipped with $\bm W_{\mu, \bk}^{\BB\rtimes \triangle}$
with $\bk =(\k_1,\ldots, \k_d) \in \RR^{d}$. As seen in Theorem \ref{thm:QOPd1=1}, the orthogonal polynomials of 
$\CV_n\big(\BB^1 \rtimes \triangle^{d}, \bm W_{\mu,\bk}^{\BB\rtimes \triangle}\big)$ are closely related to the 
classical orthogonal polynomials on the simplex $\triangle^{d+1}$. The connection extends to the reproducing
kernels, which allows us to derive the addition formula on $[-1,1] \rtimes \triangle^{d}$ from \eqref{eq:add-simplex} 
on $\triangle^{d+1}$. 

\begin{prop}  
Let $\bk =(\k_1,\ldots, \k_d)$, $\k_i \ge - 1/2$ and $\mu \ge 0$. Then, for $(x, \bm y)$ and $(u, \bm v)$ in 
$\BB^1 \rtimes \triangle^{d}$,
\begin{align} \label{eq:Pn_d=1}
 \bm P_n^{1,d}  \left(\bm W_{\mu, \bk}^{\BB\rtimes \triangle}; (x,\bm y), (u, \bm v)\right) 
 & = c_{\mu,\bk} \int_{[-1,1]^{d+2}} Z_{2n}^{|\bk| + 2\mu +d}  \big(\eta(x, \bm y, u, \bm v; \bm t) \big) \\
  &  \times \prod_{i=1}^{d} (1-t_i^2)^{\k_i-\f12}  (1-t_{d+1}^2)^{\mu - 1}(1-t_{d+2}^2)^{\mu - 1} \d \bm t, 
   \notag
\end{align}
where $c_{\mu,\bk} =  c_{\k_1}\cdots  c_{\k_{d-1}} (c_{\mu-\f12})^{2} $, and
\begin{align}\label{eq:eta()}
 \eta(x, \bm y, u, \bm v; \bm t)  = &  \sqrt{y_1 v_1} \, t_1+ \ldots + \sqrt{y_d v_d} \,t_{d} 
    + \f12 \sqrt{1-|\bm y|+x} \sqrt{1-|\bm v|+u}\,t_{d+1} \\
   & + \f12 \sqrt{1-|\bm y| - x} \sqrt{1-|\bm v|-u} \,t_{d+2}. \notag
\end{align}
Moreover,  the identity holds under the limit \eqref{eq:limitInt} if one or more $\k_i= -\f12$ and/or $\mu = 0$. 
\end{prop} 

\begin{proof} 
According to Theorem \ref{thm:QOPd1=1} and with $\bm \b$ in \eqref{eq:beta}, we have the relation
$$
  \bm P_n^{1,d}\left(\bm W_{\mu, \bk}^{[-1,1] \times \triangle}; (\bm x,\bm y), (\bm u, \bm v)\right) 
  = \bm P_n \left(\bm W_{\beta}^\triangle; \phi(\bm x,\bm y), \phi(\bm u, \bm v)\right).
$$ 
Hence, the closed-form formula for $\bm P_n\left(\bm W_{\mu, \bk}\right)$ can be deduced from the addition
formula \eqref{eq:add-simplex} on the simplex $\triangle^{d+1}$. In particular, 
\begin{align*}
  \xi_\triangle\left(\phi(x,\bm y), \phi(u, \bm v), \bm t \right) & = \sum_{i=1}^{d-1} \sqrt{y_i v_i} \, t_i + 
    \tfrac12 \sqrt{ 1-|\bm y | - x} \sqrt{ 1-|\bm v | - u} \, t_d \\
 & +   \tfrac12 \sqrt{ 1-|\bm y | + x} \sqrt{ 1-|\bm v | + u} \, t_{d+1} +
     \sqrt{y_d v_d} \,t_{d+2},
\end{align*}
where we have used $1- |\phi(x, \bm y)| = y_d$. In the statement of the theorem, we have shuffled the order of
$t_d, t_{d+1}, t_{d+2}$ of the above expression in the definition of $\eta(x, \bm y, u, \bm v, \bm t)$.
\end{proof}

Our next proposition shows that the kernel $\bm P_n^{d_1,d_2}$ can be expressed as an integral of the kernel 
$\bm P_n^{1,d_2}$ with a different parameter.  

\begin{prop} \label{prop:Pd=IntPd=1}
For $\mu \ge 0$ and $\bk =(\k_1,\ldots, \k_d)$ with $\k_i \ge -\f12$, 
\begin{align} \label{eq:P=IntP}
  & \bm P_n^{d_1,d_2} \big(\bm W_{\mu, \bk}^{\BB\rtimes \triangle}; (\bm x,\bm y), (\bm u, \bm v)\big)  \\
   = &\, c_{\mu-\f12} \int_{-1}^1 
     \bm P_n^{1,d_2}\left(\bm W_{\mu+\frac{d_1 -1}{2}, \bk}^{[-1,1]\rtimes \triangle}; \big(\zeta (\bm x,\bm y,  \bm u,\bm v; s),\bm y\big), \big(1-|\bm v|, \bm v\big) \right) 
  (1-s^2)^{\mu-1} \d s, \notag 
\end{align} 
where  
$$
\zeta (\bm x,\bm y,  \bm u,\bm v; s) = \frac{\la \bm x, \bm u \ra + s \sqrt{(1-|\bm y|)^2 -\|\bm x\|^2}  \sqrt{(1-|\bm v|)^2 -\|\bm u\|^2}}{1-|\bm v|},
$$
and $|\zeta (\bm x,\bm y,  \bm u,\bm v; s)|\le 1$ for $(\bm x, \bm y), (\bm u, \bm v) \in \BB^{d_1} \!\rtimes \triangle^{d_2}$ 
and $s \in [-1,1]$. Moreover,  the identity holds in the limit \eqref{eq:limitInt} if $\mu = 0$. 
\end{prop}

\begin{proof}
In terms of the orthogonal polynomials \eqref{eq:Qbasis} and their norms \eqref{eq:Q_OPNorm}, and the definition of 
the reproducing kernel on the unit ball, 
\begin{align*}%\label{eq:reprodPd=1}
\bm P_n^{d_1,d_2} \left(\bm W_{\mu, \bk}^{\BB\rtimes \triangle}; (\bm x,\bm y), (\bm u, \bm v)\right) & = \sum_{m=0}^n \frac{(|\bk|+ 2\mu+ d_1 + d_2)_{2m}}{(2\mu+d_1)_{2m}} \frac1{\bm h_{\bm k,m}^{\mu,\BB}  
       \bm h_{\bm j,m}^{(\bk, \a+2m),\triangle}} \\
 & \times  \sum_{\bm j = n-m} \bm T_{\bm j, n-m}^{(\bk, 2m +2\mu+d_1-1)}(\bm y) 
      \bm T_{\bm j, n-m}^{(\bk, 2m+2\mu+d_1-1)}(\bm v) \\
 & \times  (1-|\bm y|)^m (1-|\bm v|)^m \bm P_m\left(W_\mu^\BB; \frac{\bm x}{1-|\bm y|}, \frac{\bm u}{1-|\bm v|} \right). \notag 
\end{align*}
Hence, by the addition formula \eqref{eq:add-ball}, we obtain
\begin{align}\label{eq:P=IntP2}
& \bm P_n^{d_1,d_2}  \left(\bm W_{\mu, \bk}^{\BB\rtimes \triangle}; (\bm x,\bm y), (\bm u, \bm v)\right)
  = \sum_{m=0}^n \frac{(|\bk|+ 2\mu+ d_1 + d_2)_{2m}}{(2\mu+d_1)_{2m}} \frac{1}{ \bm h_{\bm j,m}^{(\bk, \a+2m),\triangle}} \\
 &\qquad \times  \sum_{\bm j = n-m} \bm T_{\bm j, n-m}^{(\bk, 2m +2\mu+d_1-1)}(\bm y) 
      \bm T_{\bm j, n-m}^{(\bk, 2m+2\mu+d_1-1)}(\bm v) \notag \\
 & \qquad \times  (1-|\bm y|)^m (1-|\bm v|)^m c_\mu \int_{-1}^1 Z_m^{\mu+\f{d_1-1}{2}}
     \left(\frac{\zeta (\bm x,\bm y, \bm u,\bm v; s)}{1-|\bm y|} \right)
  (1-s^2)^{\mu-1} \d s. \notag 
\end{align}
Now, using \eqref{eq:Qbasis_d1=1} and \eqref{eq:Q_OPNorm_d1=1}, we obtain for the case $d_1=1$ and $\a = 2\mu + d_1-1=2\mu$,
\begin{align}\label{eq:reprodPd=1}
& \bm P_n^{1,d_2}\big(\bm W_{\mu, \bk}^{\BB\rtimes \triangle}; (x,\bm y), (1-|\bm v|, \bm v)\big) 
= \sum_{m=0}^n \frac{(|\bk|+ 2\mu+ d_2 +1)_{2m}}{(2\mu+1)_{2m}}  \frac{1}{ \bm h_{\bm j,m}^{(\bk, 2\mu+2m),\triangle}} \\
& \times  \sum_{\bm j = n-m} \bm T_{\bm j, n-m}^{(\bk, 2\mu+2m)}(\bm y)  \bm T_{\bm j, n-m}^{(\bk, 2\mu+2m)}(\bm v) 
  (1-|\bm y|)^m (1-|\bm v|)^m Z_m^\mu \left(\frac{x}{1-|\bm y|}\right), \notag 
\end{align}
where we have used the relation 
$$
    \frac{1}{\sqrt{h_m^\mu}} C_m^\mu(1) C_m^\mu (t) = Z_m^\mu(t). 
$$
Setting $\mu \to \mu + \frac{d_1-1}{2}$ and $x = \zeta (\bm x,\bm y, \bm u,\bm v; s)$ in \eqref{eq:reprodPd=1}, 
and comparing its right-hand side with the right-hand side of \eqref{eq:P=IntP2} shows that \eqref{eq:P=IntP} holds.
\end{proof}

The addition formula on $\BB^{d_1} \rtimes \triangle^{d_2}$ follows from combining two previous propositions. 

\begin{thm}\label{thm:addition_BT}
Let $\mu \ge 0$ and $\k_i \ge -\f12$, $1\le i \le d_2$. For $(\bm x, \bm y)$ and $(\bm u, \bm v) \in \BB^{d_1} \! \rtimes \triangle^{d_2}$,
\begin{align}\label{eq:addition}
  \bm P_n^{d_1,d_2}  \big(\bm W_{\mu, \bk}^{\BB\rtimes \triangle},  (\bm x, \bm y), (\bm u, \bm v) \big) 
    = & \, c_{\mu,\bk} \int_{[-1,1]^{d_2+2}}   Z_{2n}^{|\bk| + 2\mu + d_1+d_2-1}
    \big( \xi(\bm x, \bm y, \bm u, \bm v; \bm s, \bm t)\big) \notag \\
   &\times \prod_{i=1}^{d_2} (1-t_i^2)^{\k_i-\f12}  \prod_{i=1}^2 (1-s_i^2)^{\mu - 1}  \d \bm t\, \d \bm s,
\end{align}
where $c_{\mu, \bk} =  c_{\k_1}\cdots c_{\k_{d_2}} (c_{\mu-\f12})^{2} $ and
\begin{align*}
& \xi(\bm x, \bm y, \bm u, \bm v; \bm s, \bm t) =  \sqrt{y_1 v_1} \, t_1 + \cdots + \sqrt{y_{d_2} v_{d_2}} \, t_{d_2}  \\ 
   & \quad + \sqrt{\frac12 \left( \la \bm x, \bm u \ra+(1-|\bm y|)(1-|\bm v|) +  s_1 \sqrt{(1-|\bm y|)^2-\|\bm x\|^2}\sqrt{(1-|\bm v|)^2-\|\bm u\|^2} \right)} \, s_2.
\end{align*}
Moreover, \eqref{eq:addition} holds under the limit \eqref{eq:limitInt} if one or more $\k_i= -\f12$ and/or $\mu = 0$. 
\end{thm}

\begin{proof}
The integral identity is deduced from \eqref{eq:P=IntP} and \eqref{eq:Pn_d=1} with $\mu$ replaced by $\mu + \frac{d-1}2$, and the 
integrand becomes 
\begin{align*}
&  \xi(\bm x, \bm y, \bm u, \bm v; \bm s, \bm t) = \eta \big(\zeta (\bm x,\bm y,  \bm u,\bm v; s_2),  
\bm y, 1-|\bm v|, \bm v;  \bm t \big)  \\ 
\qquad  &  = \sqrt{y_1 v_1} \, t_1 + \cdots + \sqrt{y_d v_d} \, t_d + \frac12 \sqrt{1-|\bm y| + 
     \zeta(\bm x, \bm y, \bm u, \bm v; s)} \sqrt{2 (1-|\bm v|)}\, t_{d+1}
\end{align*}
as the last term in \eqref{eq:eta()} becomes zero when $u = 1-|\bm v|$, which simplifies to the stated formula for
$\xi(\bm x, \bm y, \bm u, \bm v; \bm s, \bm t)$ after setting $s= s_1$ and $t_{d+1} = s_2$. 
\end{proof}

Complicated as it looks, the addition formula shows that the reproducing kernel has a one-dimensional characteristic, 
which can be used to significantly simplify the Fourier analysis on the domain $\BB^{d_1} \!\rtimes \triangle^{d_2}$. 

\subsection{Fourier orthogonal expansions}
This subsection serves as an introduction to applications of the addition formula in the study of Fourier-orthogonal
expansions. 

For $f \in L^2\big(\BB^{d_1} \!\rtimes \triangle^{d_2}, \bm W_{\mu,\bk}^{\BB\rtimes \triangle}\big)$, the 
Fourier orthogonal expansion of $f$ is defined by 
$$
   f = \sum_{n=0}^\infty \proj_n\left(\bm W_{\mu, \bk}^{\BB\rtimes \triangle}; f\right), 
$$
where $\proj_n(\bm W_{\mu, \bk}^{\BB\rtimes \triangle}):
 L^2\big(\BB^{d_1} \!\rtimes \triangle^{d_2}, \bm W_{\mu,\bk}^{\BB\rtimes \triangle}\big)
 \mapsto \CV_n\big(\BB^{d_1} \!\rtimes \triangle^{d_2}, \bm W_{\mu,\bk}^{\BB\rtimes \triangle}\big)$
denotes the orthogonal projection operator, which is an integral operator whose kernel is the reproducing kernel 
$\bm P_n\left( \bm W_{\mu,\bk}^{\BB\rtimes \triangle};  \cdot,\cdot\right)$, 
$$
 \proj_n\big(\bm W_{\mu, \bk}^{\BB\rtimes \triangle}; f, \bm x,\bm y\big)= \int_{\BB^{d_1} \!\rtimes \triangle^{d_2}}
  f (\bm u, \bm v) \bm P_n\big( \bm W_{\mu,\bk}^{\BB\rtimes \triangle};  (\bm x,\bm y), (\bm u, \bm v) \big)
   \bm W_{\mu, \bk}^{\BB\rtimes \triangle} (\bm u,\bm v) \d \bm u \d \bm v.
$$
For $n \in \NN_0$, we denote the $n$-th Fourier partial sum of the expansion by $S_n(\bm W_{\mu, \bk})$, 
$$
  S_n\big(\bm W_{\mu, \bk}^{\BB\rtimes \triangle}; f\big) 
   = \sum_{m=0}^n \proj_n\big(\bm W_{\mu, \bk}^{\BB\rtimes \triangle}; f\big).   
$$
By the standard Hilbert space theory, $S_n(\bm W_{\mu, \bk}^{\BB\rtimes \triangle}; f)$ converges to
$f$ in the norm of $L^2(\BB^{d_1}\rtimes \triangle^{d_2}, \bm W_{\mu, \bk}^{\BB\rtimes \triangle})$. 
To go beyond the $L^2$ theory, we introduce a pseudo-convolution structure defined via the addition formula \eqref{eq:addition},
which provides a tool that helps us to reduce a substantial portion of the analysis on the Fourier orthogonal expansion to the 
Fourier-Jacobi expansion in one variable. Let $w_{\a,\b}$ be the Jacobi weight function,
$$
   w_{\a,\b}(t) = (1-t)^\a (1+t)^\b, \qquad \a, \b > -1. 
$$

\begin{defn} 
For $\mu \ge 0$ and $\k_i \ge -\f12$, let $ \a = |\bk| + 2\mu + d_1+d_2-1$. For 
$f\in L^1\big(\BB^{d_1} \!\rtimes \triangle^{d_2}, \bm W_{\mu,\bk}^{\BB\rtimes \triangle}\big)$ 
and $g \in L^1([-1, 1], w_{\a, -\f12})$, define the pseudo convolution of $f$ and $g$ by
\begin{equation} \label{eq:convolu}
 f\ast g (\bm x,\bm y) = \int_{\BB^{d_1} \!\rtimes \triangle^{d_2}} f(\bm u,\bm v) T^{\mu,\bk} g\big ( (\bm x,\bm y), (\bm u,\bm v) \big)
    \bm W_{\mu, \bk}^{\BB\rtimes \triangle} (\bm u, \bm v) \d \bm u \d \bm v
\end{equation}
where the operator $g \mapsto T^{\mu,\bk} g$ is defined by 
\begin{align*}
T^{\mu,\bk} g \big ( (\bm x,\bm y), (\bm u,\bm v) \big)
 =  & \, \bm b_{\mu,\bk} \int_{[-1,1]^{d+1}} \int_{-1}^1 g \big( 2 \xi(\bm x, \bm y, \bm u, \bm v; \bm s, \bm t)^2-1 \big) \notag \\
   &\times \prod_{i=1}^{d_2-1} (1-t_i^2)^{\k_i-\f12} (1-|\bm t|^2)^{\k_{d_2}-\f12}  \prod_{i=1}^2 (1-s_i^2)^{\mu - 1}  \d \bm t\, \d \bm s,
\end{align*}
in which $\xi(\bm x, \bm y, \bm u, \bm v; \bm s, \bm t)$ is defined as in the addition formula \eqref{eq:addition}. 
\end{defn}

Using the quadratic transformation between the Jacobi polynomials and the Gegenbauer polynomials
\cite[(4.1.5) and (4.7.1)]{Sz}, 
$$
   Z_{n}^{(\a-\f12, -\f12)} (2t^2-1): = \frac{P_n^{(\a - \f12, -\f12)}(1) P_n^{(\a - \f12, -\f12)}(2t^2-1)}{h_n^{(\a-\f12, - \f12)}} =   Z_{2n}^\a (t),
$$
where $h_n^{(\a-\f12,-\f12)}$ is the squared norm of $P_n^{(\a - \f12, -\f12)}$, it follows readily that 
\begin{equation} \label{eq:kernel=*}
   \bm P_n \big(\bm W_{\mu,\bk}^{\BB\rtimes \triangle}; (\bm x, \bm y), (\bm u, \bm v) \big)
       = T^{\mu, \bk} \left(Z_n^{(\a-\f12,-\f12)} \right) \big ( (\bm x,\bm y), (\bm u,\bm v) \big).
\end{equation}
%with $ \a = |\bk| + 2 \mu + d_1+d_2 - \tfrac12$. 
As a consequence, it follows from the addition formula that  
\begin{equation} \label{eq:proj=*}
\proj_n\big(\bm W_{\mu, \bk}^{\BB\rtimes \triangle}; f, \bm x, \bm y\big) = f * Z_n^{(\a, -\f12)} (\bm x, \bm y). %, \qquad \a = |\bk| + 2 \mu + d_1+d_2 - \tfrac12.
\end{equation}

The operator $T^{(\mu, \bk)}$ has been defined and used in \cite{X05, X21b} whenever an addition formula exists. It is 
a bounded operator as shown in \cite[Lemma 6.3]{X21b}. 

\begin{lem} \label{lem:translateT}
Let $g \in L^1([-1,1],w_{\a, - \f12})$. Then, for each 
$Q_n \in \CV_n \big(\BB^{d_1} \!\rtimes \triangle^{d_2}, \bm W_{\mu,\bk}^{\BB\rtimes \triangle}\big)$, 
\begin{equation}\label{eq:FHformula}
     \int_{\BB^{d_1} \!\rtimes \triangle^{d_2}}  T^{(\mu,\bk)} g (\bm x, \bm y) Q_n (\bm y) \bm W_{\mu, \bk}^{\BB\rtimes \triangle}
       (\xb, \bm y) \d \bm x \d \bm y  = 
      \Lambda_n^{(\a,-\f12)} (g) Q_n(\bm x),
\end{equation}
where $\a = |\bk| + 2\mu + d_1+d_2-1$ and 
$$
 \Lambda_n^{(\a, -\f12)} (g) =  c_{\a, -\f12} \int_{-1}^1 g(t) \frac{P_n^{(\a, -\f12)}(t)}{P_n^{(\a,-\f12)}(1)} w_{\a,-\f12}(t)\d t.
 % \quad \hbox{with} \quad R_n^{(\a,\b)}(t) : = \frac{P_n^{(\a,\b)}(t)}{P_n^{(\a,\b)}(1)}
$$
Furthermore, for $1\le p \le \infty$ and $\bm x \in \BB^{d_1} \rtimes \triangle^{d_2}$, 
\begin{equation}\label{eq:Tbd}
  \left \| T^{(\mu,\bk)} g (\bm x, \cdot)\right\|_{L^p(\BB^{d_1} \! \rtimes \triangle^{d_2}, \bm W_{\mu, \bk}^{\BB\rtimes \triangle})} \le \|g \|_{L^p([-1,1],w_{\a,-\frac12})}.
\end{equation}
\end{lem}
 
With the boundedness of $T^{(\mu,\bk)}$ in the lemma, one can deduce \cite{X21b} the Young's inequality for the 
pseudo-convolution defined in \eqref{eq:convolu}.  

\begin{thm} \label{thm:Young}
Let $p,q,r \ge 1$ and $p^{-1} = r^{-1}+q^{-1}-1$. For $f \in L^q(\BB^{d_1} \!\rtimes \triangle^{d_2}, 
 \bm W_{\mu, \bk}^{\BB\rtimes \triangle})$ and $g \in L^r([-1,1]; w_{\a,-\f12})$ with $\a = |\bk| + 2\mu + d_1+d_2-1$,
\begin{equation} \label{eq:Young}
  \|f \ast g\|_{L^p (\BB^{d_1}\rtimes \triangle^{d_2}, \bm W_{\mu, \bk}^{\BB\rtimes \triangle})} 
     \le \|f\|_{L^q(\BB^{d_1}\rtimes \triangle^{d_2}, \bm W_{\mu, \bk}^{\BB\rtimes \triangle})}\|g\|_{L^r( [-1,1]; w_{\a,-\f12})}.
\end{equation}
\end{thm}

In particular, this shows that the operator $f \mapsto f\ast g$ is a bounded operator of 
$ L^q\big(\BB^{d_1} \!\rtimes \triangle^{d_2}, \bm W_{\mu,\bk}^{\BB\rtimes \triangle}\big)$
if $g$ is bounded in $L^1( [-1,1]; w_{\a,-\f12})$. As an application, let us consider the Ces\`aro $(C,\delta)$ means 
$S_n^\delta (\bm W_{\mu, \bk}^{\BB\rtimes \triangle};f)$ of the Fourier orthogonal series with respect to 
$\bm W_{\mu, \bk}^{\BB\rtimes \triangle}$ on $\BB^{d_1}\!\rtimes \triangle^{d_2}$. 
Let $\delta > 0$. Then the operator $S_n^\delta (\bm W_{\mu, \bk}^{\BB\rtimes \triangle})$ is defined by 
$$
 S_n^\delta \left(\bm W_{\mu, \bk}^{\BB\rtimes \triangle};f\right):= 
 \f{1}{\binom{n+\delta}{n}} \sum_{m=0}^n \binom{n-m+\delta}{n-m}\proj_n\big(\bm W_{\mu, \bk}^{\BB\rtimes \triangle}; f\big).
$$
Being a linear operator, this can be written as an integral operator with the kernel 
$\bm K_n^\delta(\bm W_{\mu, \bk}^{\BB\rtimes \triangle}; \cdot,\cdot)$,
and, by the definition of $Z_n^{(\a,-\f12)}$, and using \eqref{eq:kernel=*} and \eqref{eq:proj=*}, we see that 
$$
\bm K_n^\delta\big(\bm W_{\mu, \bk}^{\BB\rtimes \triangle}; \cdot,\cdot\big) 
       = T^{\mu, \bk} \left(k_n^{\delta, \a} \right)(\cdot,\cdot)  
 \quad \hbox{and} \quad S_n^\delta \left(\bm W_{\mu, \bk}^{\BB\rtimes \triangle};f\right) = f* k_n^{\delta,\a},
$$
where $k_n^{\delta,\a} (t) = k_n^\delta (w_{\a,-\f12}; t, 1)$, given in terms of the kernel for the $(C,\delta)$ means of the 
Fourier-Jacobi series with the weight function $w_{\a, -\f12)}$ on $[-1,1]$. The kernel $k_n^{\delta,\a}(t)$ is bounded if 
$\delta > \max\{\a, - \f12\} +\f12 = \a+\f12$ \cite[Theorem 9.1.3]{Sz} and, furthermore, the kernel is nonnegative 
if $\d \ge \a + \f32$ (cf. \cite{Ask}). Consequently, we obtain the following theorem. 
 
\begin{thm}\label{thm:cesaro}
Let $\mu \ge 0$ and $\k_i \ge -\f12$, $1 \le i\le d_2$, and $\a = |\bk| + 2\mu + d_1+d_2-1$. 
Then Ces\`aro $(C,\delta)$ means of the Fourier orthogonal expansion
for $\bm W_{\mu, \bk}^{\BB\rtimes \triangle}$ on $\BB^{d_1} \!\rtimes \triangle^{d_2}$ satisfy the following: 
\begin{enumerate}
\item If $\delta \ge \a+ \f32$, then $S_n^\delta\big(\bm W_{\mu,\bk}^{\BB\rtimes \triangle}; f\big)$ is a nonnegative operator;
\item If $\delta > \a +\f12$, then for $n= 0,1,2,\ldots$,
$$
    \left \|S_n^\delta\big( \bm W_{\mu,\bk}^{\BB\rtimes \triangle};  f\big)\right \|_{L^p(\BB^{d_1}\rtimes \triangle^{d_2},
      \bm W_{\mu,\bk}^{\BB\rtimes \triangle})} \le c\,  \|f\|_{L^p(\BB^{d_1}\rtimes \triangle^{d_2},\bm W_{\mu,\bk}^{\BB\rtimes \triangle})}, \quad 1 \le p \le \infty.
$$
\end{enumerate}
In particular, $S_n^\delta\big( \bm W_{\mu,\bk}^{\BB\rtimes \triangle};  f\big)$ converges to $f$ in 
$L^p\big(\BB^{d_1} \rtimes \triangle^{d_2}, \bm W_{\mu,\bk}^{\BB\rtimes \triangle}\big)$ if $\delta > \a + \f12$. 
\end{thm} 

We refer to \cite{X21} for several potential applications of the pseudo-convolution structure, such as weighted best 
approximation by polynomials and tight frames, which become approachable with the pseudo-convolution structure, 
but require lengthy and delicate estimates to establish highly localized kernels, and the latter can be quite technical 
and involved. 

\section{Spherical-Simplex polynomials on the surface $\SS^{d_1-1}\!\rtimes \triangle^{d_2}$}
\setcounter{equation}{0}

We consider orthogonal polynomials on the surface of $\BB^{d_1} \rtimes \triangle^{d_2}$ in this section. The 
domain is a quadratic surface, which is the wrapped product $\SS^{d_1-1}\!\rtimes \triangle^{d_2}$ of the 
unit sphere and the simplex. Orthogonal polynomials in this setting are reviewed in the first subsection,
and the addition formula is derived in the second subsection. All other known addition formulas on compact 
domains are known to be closely related to an intrinsic distance function of the domain. This is shown to be true 
in the third subsection, where a distance function is defined and justified for both $\SS^{d_1-1} \rtimes \triangle^{d_2}$ 
and $\BB^{d_1} \rtimes \triangle^{d_2}$. 

\subsection{Orthogonal polynomials on $\SS^{d_1-1}\!\rtimes \triangle^{d_2}$} 
Let $d_1$ and $d_2$ be two positive integers and $d_1 \ge 2$. We consider the surface of the domain $\BB^{d_1}\!\rtimes \triangle^{d_2}$, defined by 
$$
  \SS^{d_1-1}\!\rtimes \triangle^{d_2} = \left\{(\bm x,\bm y) \in \RR^{d_1} \times \RR_+^{d_2}: \|\bm x\| = 1- |\bm y| \le 1\right\}.
$$
For $\bm\kappa \in \RR^{d_2}$ such that $\k_i > -d_1$, $1 \le i \le d_2$ and $\g > -1$, we define  
$$
  {\sf W}_{\g,\bk}^{\SS\rtimes \triangle}(\bm y) =  \prod_{i=1}^{d_2} |y_i|^{\k_i} (1-|\bm y|)^{\g}, \qquad \yb \in \triangle^{d_2},
$$
and consider orthogonal polynomials with respect to the inner product defined by 
\begin{equation} \label{eq:ipdV0}
  \la f, g\ra_{\g, \bk}^{\SS\rtimes \triangle}  = \sb_{\g,\bk}^{\SS\rtimes \triangle} 
  \int_{ \SS^{d_1-1}\!\rtimes \triangle^{d_2}} f(\bm x,\bm y) g(\bm x,\bm y) \sW_{\g,\bk}^{\SS\rtimes \triangle}(\bm y) 
     \d \s(\bm x, \bm y), 
\end{equation}
where $\d \s(\bm x, \bm y)$ is the surface measure of $\SS^{d_1-1}\!\rtimes \triangle^{d_2}$ and
the constant $\sb_{\g,\bk}^{\SS\rtimes \triangle}$ satisfies
\begin{equation} \label{eq:b-ST0}
  \sb_{\g,\bk}^{\SS\rtimes \triangle} = \frac{1}{\o_{d_1}} \bm b^\triangle_{(\bk, \g+d_1 -1)},
\end{equation}
where $\o_d$ is the surface area of $\sph$. 
Let $\CV_n\big(\SS^{d_1-1}\!\rtimes \triangle^{d_2},  {\sf W}_{\g,\bk}^{\SS\rtimes \triangle}\big)$ be
the space of orthogonal polynomials of degree $n$ with respect to the inner product 
$\la \cdot, \cdot\ra_{\g, \bk}^{\SS\rtimes \triangle}$. 
%$$
%  \dim \CV_n\big(\SS^{d_1-1}\!\rtimes \triangle^{d_2},  {\sf W}_{\g,\bk}^{\SS\rtimes \triangle}\big)
%   = \binom{n + d_1+d_2 -1}{n} - \binom{n + d_1+d_2 -3}{n-2}.
%$$

\iffalse
\begin{rem}
We assume $d_1 > 1$ since if $d_1=1$, then the surface 
is defined by $|x| = 1-y_1-\cdots - y_d$, where we write $d = d_2$, which consists of two mirroring simplices in
$\RR^{d+1}$ joined at the line $x=0$ and $|\bm y|=1$. For $d =2$, it is two triangular planes in $\RR^3$, defined by
$\pm x = 1-y_1-y_2$, $(y_1, y_2) \in \triangle^2$, that are joined at a line and it has vertices at $(1,0,0), (-1,0,0), (0,1,0), (0,0,1)$ 
in $(x,y_1,y_2)$ coordinates. Our wrapped product construction of an orthogonal basis does not apply to this case. Moreover, for 
$d =1$, the domain becomes two line segments intersecting at one point, for which a family of orthogonal polynomials 
is studied in \cite{OX1}, and they are not eigenfunctions of a spectral operator. 
\end{rem}

Parametrising the domain $\VV^{d_1,d_2}_0$ by settign $\bm x = (1- | \bm y|) \xi$ with $\xi \in \sph$, it follows readily that 
\begin{equation} \label{eq:intV0}
  \int_{\VV_0^{d_1,d_2}} f(\bm x, \bm y) \d\s(\bm x, \bm y) =  \int_{\triangle^{d_2}} (1-|\bm y|)^{d_1} \int_{\SS^{d_1-1}}
     f\big( (1-|\bm y|) \xi, \bm y\big) \d \s(\xi)\, \d \bm y;
\end{equation}
moreover, under the same parametrization, 
$$   
  \sW_{\g,\bm \kappa}(\bm y) \d \s(\bm x, \bm y)
    = (1-|\bm y|)^{\g+ d_1-1} \sW_{(\bk,0)}^\triangle (\bm y) \d\s(\xi) \d \bm y
    =  \sW_{(\bk,\g+d_1-1)}^\triangle (\bm y) \d\s(\xi) \d \bm y.
$$
\fi
 
We again let $\{\bm T_{\bm j, m}^{\bg}: |\bm j| = m, \, \bm j\in \NN_0^{d_2}\}$ be an orthogonal basis for 
$\CV_{m}\left(\triangle^{d_2},\bm W^\triangle_{\bg}\right)$, where $\bg \in \RR^{d_2+1}$, such as the one given 
in \eqref{eq:OP_TT}. Let $\{ \sY_{\ell}^m: 1 \le \ell \le \dim \CH_m^{d_1} \}$ be an orthonormal basis of $\CH_m^{d_1}$. 
 
% and denote their normal squares by $\bm h_{\bm j,m}^\bg(\triangle)$ as in the case of the solid quadratic domain. 

\begin{prop}
For $\g > - d_1$, let $\a = \g + d_1 -1$. Define %For $m \le n$, let $\{\bm T_{\bm j, n-m}^{(\bk, 2m+ \a)}: |\bm j| = n-m, \, \bm j\in \NN_0^{d_2}\}$ be 
%an orthonormal basis for $\CV_{n-m}\left(\triangle^{d_2},\bm W^\triangle_{(\bk, \alpha+2m)}\right)$, and let $\{Y_\ell^m: 1 \le \ell \le 
%\dim \CH_m^{d_1} \}$ be an orthonormal basis of $\CH_m^{d_1}$. Let
%For $0 \le m \le n$, $\bm k \in \NN_0^{d_1}$ and $\bm j \in \NN_0^{d_2}$, and $(\bm x, \bm y) \in \VV^{d_1,d_2}$, define 
\begin{equation} \label{eq:QbasisV0}
 \sQ_{\bm j, \ell, m}^n(\bm x,\bm y) = \bm T_{\bm j, n-m}^{(\bk, \alpha+2m)}(\bm y) (1-|\bm y|)^m 
      \sY_{\ell}^m \left(\frac{\bm x}{1-|\bm y|}\right). 
\end{equation}
Then $\{\sQ_{\bm j, \ell, m}^n: |\bm j| = n-m, \, 1 \le \ell \le \dim \CH_m^{d_1}, \, 0 \le m  \le n, \bm j \in \NN_0^{d_2}\}$ 
is an orthogonal basis of $\CV_n(\SS^{d_1-1}\!\rtimes \triangle^{d_2}, \sW_{\g, \bk}^{\SS\rtimes \triangle})$. Moreover, the 
squared norm of $\sQ_{\bm j, \ell, m}^n$ 
is given by 
\begin{equation} \label{eq:Q_V0Norm}
   \sh_{m,n}^{\g,\bk}  % = \la \sQ_{\bm j, \ell, m}^n, \sQ_{\bm j,\bm k, m}^n \ra_{\mu, \bk} 
       =   \frac{\bm b^\triangle_{(\bk,\alpha)}}{\bm b^\triangle_{(\bk, \alpha + 2m)}} \bm h_{\bm j,m}^{(\bk, \a+ 2m), \triangle}
       = \frac{(\g +d_1)_{2m}} {(|\bk|+ \g + d_1 + d_2)_{2m}} \bm h_{\bm j,m}^{(\bk, \a+ 2m), \triangle}.  
\end{equation}
\end{prop}

This is \cite[Proposition 7.2.]{X26}. The most important family turns out to be when $\g = -1$ since this is 
the only case for which the orthogonal polynomials possess a spectral operator \cite{X26}. It turns out that
the addition formula holds for $\g = -1$ and only in this case. 

Let $\sP_n(\sW_{-1,\bk}; \cdot, \cdot)^{\SS\rtimes \triangle}\big)$ be the reproducing kernel of the space 
$\CV_n\big(\SS^{d_1-1}\!\rtimes \triangle^{d_2},  {\sf W}_{-1,\bk}^{\SS\rtimes \triangle}\big)$. 

\begin{thm}\label{thm:additionV0}
Let $d_1 \ge 2$ and $\k_1,\ldots, k_{d_2} \ge -\f12$. Then 
\begin{align}\label{eq:additionV0}
  \sP_n  \big(\sW_{-1, \bk}^{\SS\rtimes \triangle},  (\bm x, \bm y), (\bm u, \bm v) \big) 
    = & \, \sc_{-1,\bk} \int_{[-1,1]^{d_2+1}} Z_{2n}^{|\bk| + d_1+d_2-2}
    \big( \xi_0(\bm x, \bm y, \bm u, \bm v; s, \bm t)\big)  \\
   & \times \prod_{i=1}^{d_2} (1-t_i^2)^{\k_i-\f12}  (1-s^2)^{\f{d_1-2}{2} - 1}  \d \bm t\, \d s,\notag
\end{align}
where $ \sc_{-1,\bk} = c_{\k_1} \cdots c_{\k_{d_2}} c_{\frac{d_1-3}{2}}$ and 
\begin{align*}
& \xi_0(\bm x, \bm y, \bm u, \bm v; s, \bm t) =  \sqrt{y_1 v_1} \, t_1 + \cdots + \sqrt{y_{d_2} v_{d_2}} \, t_{d_2} 
+ \sqrt{\tfrac12 \left[ \la \bm x, \bm u \ra+(1-|\bm y|)(1-|\bm v|)\right]} \, s.
\end{align*}
Moreoever, the identity \eqref{eq:additionV0} holds under the limit if one or more $\k_i = - \f12$ or $d_1= 2$. 
\end{thm}

\begin{proof}
By \eqref{eq:QbasisV0} and \eqref{eq:Q_V0Norm} with $\g = -1$ so that $\a = d_1-2$, and using the addition formula 
for the spherical harmonics, the reproducing kernel satisfies,  
\begin{align*}
  \sP_n  \big(\sW_{-1, \bk}^{\SS\rtimes \triangle},  (\bm x, \bm y), & (\bm u, \bm v) \big) 
  =  \sum_{m=0}^n \sum_{|\bm j| = n-m} \sum_\ell \frac{(|\bk| + d_1+d_2-1)_{2m}}{(d_1-1)_{2m}}\frac{1}{\bm h_{\bm j,m}^{(\bk, \a+ 2m), \triangle}}\\
     & \times \bm T_{\bm j, n-m}^{(\bk, \alpha+2m)}(\bm y) \bm T_{\bm j, n-m}^{(\bk, \alpha+2m)}(\bm v) (1-|\bm y|)^m (1-|\bm v|)^m 
       Z_m^{\frac{d_1-2}{2}}(\la \xi,\eta \ra),
\end{align*}
where $\bm x  = (1-|\bm y|) \xi$ and  $\bm u  = (1-|\bm v|) \eta$. Comparing with \eqref{eq:reprodPd=1} with $\mu = \frac{d_1-2}{2}$, 
we obtain 
$$
\sP_n  \big(\sW_{-1, \bk}^{\SS\rtimes \triangle},  (\bm x, \bm y), (\bm u, \bm v) \big) 
 = \bm P_n^{1,d_2}\left(\bm W_{\frac{d_1-2}{2}, \bk}^{\SS\rtimes \triangle}; 
      \big((1-|\bm y|) \la \xi,\eta\ra,\bm y\big), (1-|\bm v|, \bm v)\right).  
$$
As a consequence, the integral identity follows from that of \eqref{eq:Pn_d=1} and its integrand becomes 
\begin{align*}
&  \xi_0(\bm x, \bm y, \bm u, \bm v; s, \bm t) = \eta\big( (1-|\bm y|) \la \xi,\eta\ra,\bm y, 1-|\bm v|, \bm v); (\bm t,s) \big)  \\ \qquad  & 
 = \sqrt{y_1 v_1} \, t_1 + \cdots + \sqrt{y_d v_d} \, t_d + \frac12 \sqrt{(1-|\bm y|)(1+\la \xi,\eta\ra)} \sqrt{2 (1-|\bm v|)}\, s, 
\end{align*}
which simplifies to the stated formula for $\xi_0(\bm x, \bm y, \bm u, \bm v;  s, \bm t)$ and completes the proof.
\end{proof}

As in the case of the solid domain, the addition formula shows a one-dimensional characteristic for the reproducing
kernel, which has important implications for the Fourier analysis on the surface $\SS^{d_1-1}\!\rtimes \triangle^{d_2}$. 
We need to restrict to the case $\g = -1$. 
For $f \in L^2\big(\SS^{d_1-1}\!\rtimes \triangle^{d_2},  {\sf W}_{-1,\bk}^{\SS\rtimes \triangle}$, 
the Fourier orthogonal expansion of $f$ is defined by 
$$
   f = \sum_{n=0}^\infty \proj_n \big(\sW_{-1, \bk}^{\SS\rtimes \triangle}; f\big)
$$
where $\proj_n(\sW_{-1, \bk}): L^2\big(\SS^{d_1-1}\!\rtimes \triangle^{d_2},  {\sf W}_{-1,\bk}^{\SS\rtimes \triangle}, 
\sW_{-1, \bk}\big) \mapsto \CV_n\big(\SS^{d_1-1}\!\rtimes \triangle^{d_2},  {\sf W}_{-1,\bk}^{\SS\rtimes \triangle}\big)$
denotes the orthogonal projection operator, which is an integral operator with the reproducing kernel 
$\sP_n\left( \sW_{-1,\bk}^{\SS\rtimes \triangle};  \cdot,\cdot\right)$ as its kernel.  
\iffalse
$$
  \proj_n\big(\sW_{-1, \bk}^{\SS\rtimes \triangle}; f, \bm x,\bm y)= 
  \int_{\SS^{d_1-1}\!\rtimes \triangle^{d_2}} f (\bm u, \bm v) \sP_n\big( \sW_{-1,\bk};  (\bm x,\bm y), (\bm u, \bm v) \big)
   \sW_{-1, \bk}^{\SS\rtimes \triangle} (\bm u,\bm v) \d \s(\bm u, \bm v).
$$
%For $n \in \NN_0$, we denote the $n$-th Fourier partial sum of the above expansion by $S_n^{\mu, \bk}$, 
%$$
%  S_n^{\mu, \bk}f  = \sum_{m=0}^n \proj_k^{\mu, \bk} f.   
%$$
\fi

Similarly to the solid domain $\BB^{d_1} \! \rtimes \triangle^{d_2}$, the addition formula leads to a pseudo-convolution 
structure. % that is an analog of the one defined on the solid domain. 

\begin{defn} 
For $\k_i \ge -\f12$, $1 \le i \le d$, let $ \a = |\bk| + d_1+d_2-2$. For $f\in L^1\big(\SS^{d_1-1}\!\rtimes \triangle^{d_2}\big)$ 
and $g \in L^1([-1, 1], w_{\a, -\f12})$, define the pseudo convolution of $f$ and $g$ by
\begin{equation} \label{eq:convolu0}
 f\ast_0 g (\bm x,\bm y) = \int_{\BB^{d_1} \! \rtimes \triangle^{d_2}} f(\bm u,\bm v) \sT^{\bk} g\big ( (\bm x,\bm y), (\bm u,\bm v) \big)
    \sW_{-1, \bk}^{\SS\rtimes \triangle} (\bm u, \bm v) \d \s(\bm u, \bm v)
\end{equation}
where the operator $g \mapsto \sT^{\bk} g$ is defined by 
\begin{align*}
\sT^{\bk} g \big ( (\bm x,\bm y), (\bm u,\bm v) \big)
 =  & \, \sb_{-1,\bk} \int_{[-1,1]^{d+1}} \int_{-1}^1 g \big( 2 \xi_0(\bm x, \bm y, \bm u, \bm v; s, \bm t)^2-1 \big) \notag \\
   &\times \prod_{i=1}^{d_2} (1-t_i^2)^{\k_i-\f12}  (1-s^2)^{\f{d_1-2}{2} - 1}  \d \bm t\, \d s.
\end{align*}
\end{defn}

From the addition formula, the definition implies readily that  
\begin{equation*}% \label{eq:proj=*0}
\proj_n\big(\sW_{-1,\bk}^{\SS\rtimes \triangle}; f, \bm x, \bm y\big) =
      f *_0 Z_n^{(\a, -\f12)} (\bm x, \bm y). %, \qquad \a = |\bk| + 2 \mu + d_1+d_2 - \tfrac12.
\end{equation*}
Under this pseudo-convolution, analogous results to those on $\BB^{d_1} \!\rtimes \triangle^{d_2}$ can all be established,
including Lemma \ref{lem:translateT} and Theorem \ref{thm:Young}, for example. We state only the analog of Theorem \ref{thm:cesaro}.

\iffalse Moreover, we also have an analog of Lemma \ref{lem:translateT}, which shows the bouhndeness of the operator $\sT^\bk g$
and, as a consequence, that the pseudo-convolution satisfies Young's inequality. 

\begin{thm} \label{thm:Young0}
Let $p,q,r \ge 1$ and $p^{-1} = r^{-1}+q^{-1}-1$. For $f \in L^q(\VV_0^{d_1,d_2}, \sW_{-1, \bk})$ and
$g \in L^r([-1,1]; w_{\a,-\f12})$ with $\a = |\bk| + d_1+d_2-2$,
\begin{equation} \label{eq:Young0}
  \|f \ast g\|_{L^p (\VV_0^{d_1,d_2}, \sW_{-1, \bk})} \le \|f\|_{L^q(\VV_0^{d_1,d_2}, \sW_{-1, \bk})}\|g\|_{L^r( [-1,1]; w_{\a,-\f12})}.
\end{equation}
\end{thm}

Finally, we state an analog of Theorem \ref{thm:cesaro} on the Ces\`aro $(C, \delta)$ means $S_n^\delta \left(\sW_{-1, \bk};f\right)$
of the Fourier orthogonal expansion, 
$$
 S_n^\delta \left(\sW_{-1, \bk};f\right) = \f{1}{\binom{n+\delta}{n}} \sum_{m=0}^n \binom{n-m+\delta}{n-m}\proj_n(\sW_{-1, \bk}; f),
$$
which is equal to $f *_0 k_n^{\delta,\a}$, by \eqref{eq:proj=*0}, with $k_n^{\delta,\a}(t)= k_n^\delta (w_{\a,-\f12}; t, 1)$ of the 
Fourier-Jacobi series, as in the case of the solid domain. 
 \fi
 
\begin{thm}\label{thm:cesaro_0}
Let $\k_i \ge -\f12$, $1 \le i\le d_2$, and $\a = |\bk| + d_1+d_2-2$. Then Ces\`aro $(C,\delta)$ means of the Fourier orthogonal 
expansion for $\sW_{-1, \bk}^{\SS \rtimes \triangle}$ on $\SS^{d_1-1}\!\rtimes \triangle^{d_2}$ satisfie the following: 
\begin{enumerate}
\item If $\delta \ge \a+ \f32$, then $S_n^\delta\big(\sW_{\mu,\bk}^{\SS\rtimes \triangle}; f\big)$ is a nonnegative operator;
\item If $\delta > \a +\f12$, then for $n= 0,1,2,\ldots$,
$$
    \left \|S_n^\delta\big( \sW_{-1,\bk}^{\SS\rtimes \triangle};  f\big)\right \|_{L^p\big(\SS^{d_1-1}\!\rtimes \triangle^{d_2},\sW_{-1,\bk}^{\SS\rtimes \triangle}\big)} \le 
     \|f\|_{L^p\big(\SS^{d_1-1}\!\rtimes \triangle^{d_2},\sW_{-1,\bk}^{\SS\rtimes \triangle}\big)}, \qquad 1 \le p \le \infty.
$$
\end{enumerate}
In particular, $S_n^\delta\big( \sW_{-1,\bk}^{\SS\rtimes \triangle};  f\big)$ converges to $f$ in
$L^p\big(\SS^{d_1-1}\!\rtimes \triangle^{d_2},\sW_{-1,\bk}^{\SS\rtimes \triangle}\big)$ if $\delta > \a + \f12$. 
\end{thm} 

\subsection{Distance functions on quadratic domains} 

It is folklore in approximation theory that an appropriate distance function on a compact domain with boundary 
should take into account the positions of the points. For example, an appropriate distance function for 
the unit ball $\BB^d$ is defined by 
$$
  \sd_{\BB}(\bm x, \bm y) =  \arccos \left[ \la \bm x, \bm y\ra + \sqrt{1- \| \bm x\|^2} \sqrt{1- |\bm y\|^2}\right], \qquad \bm x, \bm y \in \BB^d. 
$$
If two pairs of points have the same distance under this metric, one pair is located close to the boundary, and the other pair is 
located close to the center of the disk $\BB^2$, then the first pair will appear much closer than the second pair if one plotted them
on paper. Such an intrinsic distance is important for various questions in analysis on the compact domain. For example, the 
doubling weight function or doubling measure should be defined in terms of such a distance function 
(see, for example, \cite{DaiX, X21} and references therein).  

The distance $\sd_\BB$ on the unit ball agrees with the geodesic distance of the unit sphere $\SS^d$ of $\RR^{d+1}$, 
which has no boundary, defined by 
$$
\sd_{\SS} (\xi, \eta) = \arccos \la \xi, \eta\ra, \qquad \xi, \eta \in \SS^d,
$$
if we let $\xi = \big(\bm x, \sqrt{1-\|\bm x\|^2}\big)$ and $\eta = \big(\bm y, \sqrt{1-|\bm y\|^2}\big)$. 
As can be seen from \eqref{eq:additionF} and \eqref{eq:add-ball}, these two distance functions are 
closely related to the addition formula on $\SS^d$ and $\BB^d$. 

The same phenomenon appears for our new domains as well. We start with the domain $\SS^{d_1-1}\!\rtimes \triangle^{d_2}$, 
and define a distance function suggested by the addition formula \eqref{eq:additionV0}.  

\begin{prop}
Let $d_1 \ge 2$. For $(\bm x, \bm y), (\bm u,\bm v) \in \SS^{d_1-1}\!\rtimes \triangle^{d_2}$, define 
$$ 
  \sd_{\SS\rtimes \triangle}\big((\bm x, \bm y), (\bm u,\bm v) \big) = \arccos \left[ \sqrt{y_1 v_1} + \cdots + \sqrt{y_d v_d} 
     + \sqrt{\frac{\la \bm x,\bm u\ra +  (1-|\bm y|)(1- |\bm v|)}{2}} \right].
$$ 
Then $\sd_{\SS\rtimes \triangle}(\cdot,\cdot)$ defines a distance function on $\SS^{d_1-1}\!\rtimes \triangle^{d_2}$. 
\end{prop}
 
\begin{proof}
By its definition, $\sd_{\SS\rtimes \triangle} (\cdot,\cdot)$ is symmetric in its variables. 
Since $\|\bm x\| = 1-|\bm y|$ and $\|\bm u\| = 1-|\bm v|$, we obtain
$$
0 \le \sqrt{y_1 v_1} + \cdots + \sqrt{y_d v_d} + \sqrt{\frac{\la \bm x,\bm u\ra +
   (1-|\bm y|)(1-|\bm v|)}{2}}  \le 1
$$
by the Cauchy-Schwarz inequality, so that $\sd_{\SS\rtimes \triangle} \big((\bm x, \bm y), (\bm u,\bm v) \big) \ge 0$ for all 
$\big((\bm x, \bm y), (\bm u,\bm v) \big) \in \SS^{d_1-1}\!\rtimes \triangle^{d_2}$. Furthermore, 
$\sd_{\SS\rtimes \triangle} \big((\bm x, \bm y), (\bm x,\bm y) \big) =
\arccos 1 = 0$. Thus, it remains to verify that $\sd_{\SS\rtimes \triangle} (\cdot,\cdot)$ satisfies the triangle inequality. 
Let $F: \SS^{d_1-1}\!\rtimes \triangle^{d_2} \mapsto \RR^{d_1 + d_2+1}$ be defined by
$$
  F(\bm x, \bm y) = \left(\frac{\bm x}{ \sqrt{2(1-|\bm y|)}}, \frac{\sqrt{1-|\bm y|}}{\sqrt{2}}, \sqrt{y_1}, \cdots, \sqrt{y_{d_2}}\right).
$$
Using $\|\bm x\| = 1- |\bm y|$, it is easy to see that $\|F(\bm x, \bm y)\|^2 = 1$, so that $F(\bm x, \bm y) \in \SS^{d_1+d_2}$. 
Moreover, for $(\bm x, \bm y), (\bm u, \bm v) \in \SS^{d_1-1}\!\rtimes \triangle^{d_2}$, 
setting $\bm x = (1-|\bm y|) \xi$ and $\bm u = (1-|\bm v|) \eta$, then
\begin{align*}
   \la F(\bm x, \bm y),  F(\bm u, \bm v) \ra 
       & =\frac{\la \bm x, \bm u\ra }{2 \sqrt{1-|\bm y|}\sqrt{1- |\bm v|}} + \frac{\sqrt{1-|\bm y|}\sqrt{1- |\bm v|}}{2}+ \sqrt{y_1 v_1}+ \cdots+
         \sqrt{y_d v_d} \\
      & =   \frac{\sqrt{1-|\bm y|}\sqrt{1- |\bm v|} (1+\la \xi, \eta\ra)}{2} + \sqrt{y_1 v_1}+ \cdots+ \sqrt{y_d v_d} \\
      & =  \sqrt{ \frac{\la \bm x, \bm y \ra+ (1-|\bm y|) (1- |\bm v|)}{2}}   + \sqrt{y_1 v_1}+ \cdots+ \sqrt{y_d v_d}.
\end{align*}
In particular, this implies immediately that 
$$
  \sd_{\SS\rtimes \triangle}\big((\bm x, \bm y), (\bm u,\bm v) \big) = \sd_{\SS} \big ( F(\bm x, \bm y), F(\bm u, \bm v)\big),
$$
where $\sd_\SS$ is the distance function of $\SS^{d_1+d_2}$. As a consequence, the triangle inequality 
of $\sd_{\SS\rtimes \triangle}$ follows from that of $\sd_\SS$. This completes the proof. 
\end{proof} 
 
Our next proposition gives a distance function on the solid domain $\BB^{d_1}\!\rtimes \triangle^{d_2}$. 

\begin{prop}
Let $d_1, d_2 \ge 1$. For $(\bm x, \bm y), (\bm u,\bm v) \in \BB^{d_1}\!\rtimes \triangle^{d_2}$, define 
\begin{align*}
 & \sd_{\BB\rtimes \triangle}\big((\bm x, \bm y), (\bm u,\bm v) \big)  =   \arccos \Big( \sqrt{y_1 v_1} + \cdots + \sqrt{y_d v_d}  \\
     & \qquad
     +   \sqrt{\frac12 \left( \la \bm x,\bm u\ra + (1-|\bm y|)(1- |\bm v|))+ \sqrt{(1-|\bm y|)^2- \|\bm x\|^2} \sqrt{(1-|\bm v|)^2- \|\bm u\|^2}\right)}\,  \Bigg ). 
\end{align*} 
Then $\sd_{\BB\rtimes \triangle}(\cdot,\cdot)$ defines a distance on $\BB^{d_1}\!\rtimes \triangle^{d_2}$. 
\end{prop}

\begin{proof}
Taking a cue from the relation between the distance functions for $\BB^d$ and $\SS^d$, for $(\bm x, \bm y)$
and $(\bm u, \bm v)$ in $\BB^{d_1}\!\rtimes \triangle^{d_2}$, we introduce the notation
$$
  \bm X =  \left(\bm x, \sqrt{(1-|\bm y|)^2-\|\bm x\|^2}\right), \qquad  \bm U =  \left(\bm u, \sqrt{(1-|\bm v|)^2-\|\bm u\|^2}\right)
$$
and observe that $(\bm X, \bm y), (\bm U, \bm v) \in \SS^{d_1}\!\rtimes \triangle^{d_2}$. Our definition then shows
$$
  \sd_{\BB\rtimes \triangle}\big((\bm x, \bm y), (\bm u,\bm v) \big) = \sd_{\SS^{d_1}\!\rtimes \triangle^{d_2}} \big((\bm X, \bm y), (\bm U, \bm v)\big), 
$$
where we have used $\sd_{\SS\!\rtimes \triangle} = \sd_{\SS^{d_1}\!\rtimes \triangle^{d_2}}$ to indicate the dependence on 
dimensions. From this identity follows readily that $\sd_{\BB\rtimes \triangle}$ is indeed an appropriate distance on 
$\BB^{d_1}\!\rtimes \triangle^{d_2}$.
\end{proof}

One can see the trace of the distance function $\sd_{\BB\rtimes \triangle}$ in the addition formula \eqref{eq:addition}, just as 
one can see the trace of $\sd_{\SS\rtimes \triangle}$ in the addition formula \eqref{eq:additionV0}, or $\sd_\BB$ from the 
addition formula \eqref{eq:add-ball}. 

The spectral operator and the addition formula provide tools for analysis on the quadratic domains 
$\BB^{d_1}\!\rtimes \triangle^{d_2}$ and $\SS^{d_1-1}\!\rtimes \triangle^{d_2}$. In particular, they make 
it possible to work in the framework developed in \cite{X21} on the space of homogeneous type 
$\left(\BB^{d_1}\!\rtimes \triangle^{d_2}, \sW, \sd_{\SS\rtimes \triangle}\right)$ and 
$\left(\SS^{d_1-1}\!\rtimes \triangle^{d_2}, \sW, \sd_{\SS \rtimes \triangle}\right)$ with the distance
functions we just defined and the doubling weight $\sW$ and $\bm W$ defined via these distance functions. 

\section{Ball-Laguerre polynomials on the domain $\BB^{d_1} \rtimes \RR_\Sigma^{d_2}$}
\setcounter{equation}{0}

Recall that $\RR_\Sigma^d$ is the domain for alternative product Laguerre polynomials defined in \eqref{eq:R-Sig}.  
We work with orthogonal polynomials on $\BB^{d_1} \rtimes \RR_\Sigma^{d_2}$, which are reviewed in the first 
subsection, and derive the addition formula for the reproducing kernels in the second subsection. 

\subsection{Orthogonal polynomials}
Let $d_1$ and $d_2$ be two positive integers. Let $\BB^{d_1} \rtimes \RR_\Sigma^{d_2}$ be the quadratic 
domain in $\RR^{d_1+d_2}$ defined by
\begin{align*}
\BB^{d_1} \! \rtimes \RR_\Sigma^{d_2} %\, & = \left\{(\xb,\yb): \frac{\xb}{|\yb|} \in \BB^{d_1}, \quad \yb \in \RR_\Sigma^{d_2}\right\} 
        = \left\{(\xb,\yb) \in \BB^{d_1} \times \RR_\Sigma^{d_2}:  \| \xb \| \le |\yb|\right\}.
\end{align*}
For $\mu > -\f12$ and $\bk \in \RR^{d_2-1}$ with $\k_i > -1$, $1\le i \le d_2-1$, we define 
\begin{align*}
\Wb_{\mu, \kb}^{\BB\rtimes \Sigma}(\xb,\yb)  = \left[ |\bm y|^2- \|\bm x\|^2 \right]^{\mu - \f12} 
\prod_{i=1}^{d_2-1} y_i^{\k_i} \e^{-|\yb'| -|\yb|}
\end{align*}
for $(\xb,\yb) \in \BB^{d_1}\! \rtimes \RR_\Sigma^{d_2}$ and consider orthogonal polynomials with respect to the 
inner product defined by 
\begin{equation} \label{eq:ipdB-Sig}
  \la f, g\ra_{\mu, \bk}^{\BB\rtimes \Sigma}  = \bm b_{\mu,\bk}^{\BB\rtimes \Sigma}
   \int_{{\BB\rtimes \RR_\Sigma}} f(\bm x,\bm y) g(\bm x,\bm y) 
  \bm W^{\BB\rtimes \Sigma}_{\mu,\bk}(\bm x,\bm y) \d \bm x \d \bm y, 
\end{equation}
where $ \bm b_{\mu,\bk}^{\BB\rtimes \Sigma}$ is the normalization constant so that $ \la 1, 1\ra_{\mu, \bk} =1$
and it satisfies 
\begin{equation} \label{eq:b_B-Sig}
  \bm b_{\mu,\bk}^{\BB\rtimes \Sigma} = \bm b^\BB_\mu \times \bm b^L_{\bg} 
       \quad\hbox{with}\quad \bm \gamma = (\bk, d_1+ 2\mu -1) \in \RR^{d_2} 
\end{equation}
in terms of the normalization constant $\bm b_\mu^\BB$ in \eqref{eq:bB} for $\bm W_\mu^\BB$ and the constant
$\bm b_{\bm \gamma}^\Sigma$ of $\bm W_{\bm \gamma}^\Sigma$. 

Orthogonal polynomials on $\BB^{d_1} \rtimes \RR_\Sigma^{d_2}$ equipped with $\bm W_{\mu,\bk}^{\BB\rtimes \Sigma}$
are studied recently in \cite{X26} for all $d_1, d_2 \ge 1$, whereas the case $(d_1,d_2) = (d,1)$ was first studied in 
\cite{X20}, for which the domain is an infinite rotating cone in $\RR^{d+1}$. 
Let $\CV_n\big(\BB^{d_1} \rtimes \RR_\Sigma^{d_2}, \bm W_{\mu,\bk}^{\BB\rtimes \Sigma}\big)$ be the 
space of orthogonal polynomials of degree $n$ with respect to the inner product $\la \cdot, \cdot\ra_{\mu, \bk}$. 

Let $\{\bm B_{\bm k, m}^{\mu}: |\bm k|= m, \, \bm k \in \NN_0^{d_1}\}$ be an orthogonal basis for 
$\CV_m(\BB^{d_1},\bm W^\BB_\mu)$, and let $\{\hat \Lb_{\bm j, m}^{\bg}: |\bm j| = m, \, \bm j\in \NN_0^{d_2}\}$ 
be an orthogonal basis for $\CV_{m}\left(\RR_\Sigma^{d_2},\bm W^\Sigma_{\bg}\right)$. 
We denote by $\bm h_{\bm k,m}^{\mu,\BB}$ and $\bm h_{\bm j,m}^{\bg,\Sigma}$ the square norms
$$
  \bm h_{\bm k,m}^{\mu, \BB} =  \langle \bm B_{\bm k, m}^{\mu}, \bm B_{\bm k, m}^{\mu}\rangle_{\mu}^\BB
 \quad and \quad \bm h_{\bm j,m}^{\bg,\Sigma} = \big \langle \hat \Lb_{\bm j, m}^{\bg}, \hat \Lb_{\bm j, m}^{\bg} \big \rangle_{\bg}^\Sigma.
$$

\begin{prop}
For $\mu > -\f12$, $\bk\in \RR^{d_2-1}$ with $\k_j > -1$, $1 \le j \le d_2-1$, define
\begin{equation} \label{eq:OP_BSig}
 \bm Q_{\bm j, \bm k, m}^n(\bm x,\bm y) = \hat \Lb_{\bm j, n-m}^{(\bk, \alpha+2m)}(\bm y) |\bm y|^m 
      \bm B_{\bm k, m}^\mu \left(\frac{\bm x}{|\bm y|}\right),
\end{equation}
where $\alpha = 2 \mu + d_1-1$. Then $\{\bm Q_{\bm j, \bm k, m}^n: |\bm k| = m, \, |\bm j| = n-m, \, \bm k \in \NN_0^{d_1}, \, \bm j \in \NN_0^{d_2}, 0\le m \le n\}$ 
is an orthogonal basis of $\CV_n\big(\BB^{d_1}\!\! \rtimes \RR_\Sigma^{d_2}, \bm W_{\mu, \bk}^{\BB\rtimes \Sigma}\big)$. 
Moreover, the squared norm of $\bm Q_{\bm j, \bm k, m}^n$ is given by 
\begin{equation} \label{eq:OP_BSig_Norm}
   \bm h_{\bm j, \bm k, m}^{n; \mu,\bk}  %= \la \bm Q_{\bm j, \bm k, m}^n, \Qb_{\bm j,\bm k, m}^n \ra_{\mu, \bk} 
     %  =   \frac{\bm b^\triangle_{(\bk,\alpha)}}{\bm b^\triangle_{(\bk, \alpha + 2m)}}
       =  (2\mu+d_1)_{2m} \bm h_{\bm k,m}^{\mu, \BB}  
       \bm h_{\bm j,m}^{(\bk, \a+2m), \triangle}. 
\end{equation}
\end{prop}
  
This is Proposition 5.1 of \cite{X26} and the case $d_2 = 1$ was first studied in \cite{X20}. 
The case $(d_1, d_2)  = (1,d)$ is particularly important for our study, for which the domain becomes 
$$
 [-1,1] \rtimes \RR_\Sigma^d = \left \{(x, \yb) \in \RR \times \RR_\Sigma^d: |x| \le |\yb| \right\},
$$
which is affine equivalent to $\RR_+^{d+1}$ under $\yb \mapsto \vb = (y_1,\ldots, y_{d-1}, \frac12(|\yb|- x), \frac12 (|\yb|+x))$.

\begin{thm} \label{thm:Q_BSig_d1=1}
The mapping $(x, \bm y) \in [-1,1] \rtimes \RR_\Sigma^{d} \mapsto \bm u \in \RR_+^{d+1}$ defined by
$$
  \bm u = \phi(x, \bm y) = \left(y_1, \ldots, y_{d-1}, \, \tfrac12(|\bm y|-x),  \tfrac12(|\bm y|+x) \right)
$$
is a bijection between $\CV_n\big([-1,1]\rtimes \RR_\Sigma^{d_2}, \bm W_{\mu,\bk}^{\BB\rtimes \Sigma}\big)$ 
and  $\RR_+^{d+1}$. Moreover, under this linear transformation, the weight function 
$\bm W_{\mu,\bk}^{\BB\rtimes \Sigma}(x, \bm y)$ with $\bk = (\k_1,\ldots, \k_d)$ becomes
$$
    \bm W_{\bm \b}^L (\ub) = 2^{2\mu-1} u_1^{\k_1} \cdots u_{d-1}^{\k_{d-1}} \ub_d^{\mu-\f12} u_{d+1}^{\mu -\f12}
      \e^{-|\ub|},
$$
where ${\bm \beta} = (\k_1,\ldots, \k_{d-1}, \mu-\f12,\mu-\f12) \in \RR^{d+1}$. In particular, let $\{\bm L_{\bm j, n}^{\bm \b}: 
|\bm j|  = n, \, \bm j \in \NN_0^{d+1}\}$ be an orthogonal basis of $\CV_n(\RR_+^{d+1}, \bm W^L_{\bm \b})$ on 
the $\RR_+^{d+1}$. Then 
\begin{equation} \label{eq:OP_B-Sig_d1=1}
   \bm Q_{\bm j, n}^{\mu, \bk} (x,\bm y) = \bm L_{\bm j, n}^{\bm \b}\left(y_1, \ldots,  y_{d-1}, 
       \tfrac12(|\bm y|-x),  \tfrac12(|\bm y|+x) \right)
\end{equation}
for $|\bm j| = n$ consist of an orthogonal basis for $\CV_n\big([-1,1]\rtimes \RR_\Sigma^d, \bm W_{\mu, \bk}^{[-1,1] \rtimes \Sigma}\big)$. 
\end{thm}

\begin{proof}
The inverse of the linear transform $\ub \mapsto (x, \bm y)$ is easily seen as given by 
$$
  x= u_{d+1}-u_d, \, y_1 = u_1, \ldots, y_{d-1} = u_{d-1}, \, y_d = |\bm u| - 2 (u_1+ \cdots + u_{d-1}). 
$$
It follows that $u_d+u_{d+1} = |\bm y|$, so that $|x| \le |\bm y|$ for $(x, \bm y) \in [-1,1]\rtimes \RR_\Sigma^d$ 
becomes $|u_d- u_{d+1}| \le u_d+u_{d+1}$, which is equivalent to $u_d \ge 0$ and $u_{d+1} \ge 0$, so $\ub \in \RR_+^{d+1}$. 
Furthermore, $|\ub| = |\yb| + |\yb'|$, so that 
$$
  \bm W_{\mu,\bm \k}^{[-1,1]\rtimes \Sigma}(x, \bm y) = \prod_{i=1}^d y_i^{\k_i} (|\bm y|-x)^{\mu-\f12} 
      (|\bm y|+x)^{\mu-\f12} \e^{-|\yb|- |yb'|} = 2^{2\mu-1} \bm W_{\bm \b}^L (\ub)  
$$
with $\bm u \in \RR_+^{d+1}$. Thus, an orthogonal basis for $\CV_n\big([-1,1]\rtimes \RR_\Sigma^{d_2}, 
\bm W_{\mu,\bk}^{\BB\rtimes \Sigma}\big)$ can be derived from that of $\CV_n(\RR_+^{d+1}, \bm W^L_{\bm \b})$.
\end{proof}

\subsection{Addition formula} 
We establish an addition formula for the orthogonal polynomials on $\BB^{d_1} \rtimes \RR_\Sigma^{d_2}$. 
In terms of the orthogonal basis in \eqref{eq:OP_BSig}, the kernel is given by
\begin{equation}\label{eq:reprodP_BSig}
\bm P_n^{d_1,d_2}\! \left(\bm W_{\mu, \bk}^{\BB\times \Sigma}; (\bm x,\bm y), (\bm u, \bm v)\right) = \sum_{m=0}^n 
  \sum_{|\bm j| = n-m} \sum_{|\bm k| = m} \frac{\bm Q_{\bm j, \bm k, m}^n(\bm x,\bm y)\bm Q_{\bm j, \bm k, m}^n(\bm u,\bm v)}{
   \la \bm Q_{\bm j, \bm k, m}^n,\bm Q_{\bm j, \bm k, m}^n\ra_{\mu, \bk}^{\BB\rtimes \Sigma} }.
\end{equation}
Like the case in $\BB^{d_1} \rtimes \triangle^{d_2}$, we first consider the case $(d_1,d_2) = (1,d)$, for which the
orthogonal polynomials can be expressed in terms of the product Laguerre polynomials by \eqref{eq:OP_B-Sig_d1=1}.
In particular, it follows from 

\begin{prop}  \label{prop:Pn_BSig_d1=1}
Let $\bk =(\k_1,\ldots, \k_d)$, $\k_i > - 1/2$ and $\mu > 0$. Then, for $(x, \bm y)$ and $(u, \bm v)$ in 
$\BB^1 \rtimes \RR_\Sigma^{d} = [-1,1] \rtimes \RR_\Sigma^{d}$,
\begin{align*} 
 & \bm P_n^{1,d} \left(\bm W_{\mu, \bk}^{[-1,1] \rtimes \Sigma}; (x,\bm y), (u, \bm v)\right) 
  = \cb_{\mu,\bk} \int_{[0,\pi]^{d+1}} L_n^{|\bk| + 2\mu +d-1} \big(\eta(x, \bm y, u, \bm v; \bm \t) \big) \\
  & \times \prod_{i=1}^{d-1} j_{\k_i-\f12}\big ( \sqrt{y_i v_i} \sin \t_i\big) j_{\mu-\f12} 
   \Big(\sqrt{\tfrac{|\yb|+x}2} \sqrt{\tfrac{|\vb|+u}{2}} \sin \t_{d} \Big) 
   j_{\mu-\f12} \Big(\sqrt{\tfrac{|\yb|-x}2} \sqrt{\tfrac{|\vb|-u}{2}} \sin \t_{d+1} \Big) \\
 &\times \exp \left( -\sum_{i=1}^{d-1} \sqrt{y_i v_i} \cos \t_i- \f12 \sqrt{(|\yb|+x) (|\vb|+u)} \cos \t_{d} - \f12 \sqrt{(|\yb|-x) (|\vb|-u)} 
 \cos \t_{d+1}\right)\\
  & \times \prod_{i=1}^{d-1} (\sin \t_i)^{\k_i} (\sin \t_d)^{2\mu-1} (\sin \t_{d+1})^{2\mu-1} \d \bm{\t},
\end{align*}
where $\cb_{\mu,\bk} =  \cb_{\bk, \mu-\f12, \mu - \f12}^L$, and
\begin{align}\label{eq:eta()BSig}
 \eta(x, \bm y, u, \bm v; \bm t)  = & |\yb'| + |\yb| +|\vb'| + |\vb| + 2 \sum_{i=1}^{d-1} \sqrt{y_i v_i} \cos \t_i \\
   & +  \sqrt{(|\yb|+x) (|\vb|+u)} \cos \t_{d} + \sqrt{(|\yb|-x) (|\vb|-u)} \cos \t_{d+1}. \notag
\end{align}
\end{prop} 

This follows immediately from Theorem \ref{thm:Q_BSig_d1=1}, in particular, \eqref{eq:OP_B-Sig_d1=1} which
shows that 
$$
\bm P_n^{1,d} \left(\bm W_{\mu, \bk}^{\BB\rtimes \Sigma}; (x, \yb), (u,\vb)\right) 
= \bm P_n\left(\bm W_{\bm \b}^L; \phi(x, \yb), \phi(u, \vb) \right),
$$ 
so that the stated identity follows from the addition formula in \eqref{eq:addL}. 

Our next proposition is an analog of Proposition \ref{prop:Pd=IntPd=1}, which relates the reproducing
kernel for $(d_1,d_2)$ to that of $(1, d_2)$. 

\begin{prop} \label{prop:Pd=IntPd_BSig}
For $\mu > 0$ and $\bk =(\k_1,\ldots, \k_d)$ with $\k_i > -\f12$, 
\begin{align} \label{eq:P=IntPBSig}
  & \bm P_n^{d_1,d_2} \big(\bm W_{\mu, \bk}^{\BB\rtimes \Sigma}; (\bm x,\bm y), (\bm u, \bm v)\big)  \\
   = &\, c_{\mu-\f12} \int_{-1}^1 
     \bm P_n^{1,d_2}\left(\bm W_{\mu+\frac{d_1 -1}{2}, \bk}^{[-1,1]\rtimes \Sigma}; 
         \big(\zeta (\bm x,\bm y,  \bm u,\bm v; s),\bm y\big), \big(|\bm v|, \bm v\big) \right) 
  (1-s^2)^{\mu-1} \d s, \notag 
\end{align} 
where  
$$
\zeta (\bm x,\bm y,  \bm u,\bm v; s) = \frac{\la \bm x, \bm u \ra + s \sqrt{|\bm y|^2 -\|\bm x\|^2}\sqrt{|\bm v|^2 -\|\bm u\|^2}}{|\bm v|},
$$
and $|\zeta (\bm x,\bm y,  \bm u,\bm v; s)|\le 1$ for $(\bm x, \bm y), (\bm u, \bm v) \in \BB^{d_1}\! \rtimes \RR_\Sigma^{d_2}$ 
and $s \in [-1,1]$. 
\end{prop}

The proof of this proposition follows almost verbatim from that of Proposition \ref{prop:Pd=IntPd=1}; we shall
omit the details. We are now ready to state the addition formula for 
$\Pb_n^{d_1,d_2} \big(\bm W_{\mu, \bk}^{\BB\rtimes \Sigma}; \cdot,\cdot)$.

\begin{thm}
Let $\bk =(\k_1,\ldots, \k_d)$, $\k_i > - 1/2$ and $\mu > 0$. Then, for $(\xb, \bm y)$ and $(\ub, \bm v)$ in 
$\BB^{d_1} \rtimes \RR_\Sigma^{d_2}$,
\begin{align*} 
 & \bm P_n^{d_1,d_2} \left(\bm W_{\mu, \bk}^{\BB\rtimes \Sigma}; (x,\bm y), (u, \bm v)\right) 
  = \cb_{\mu,\bk} \int_{[0,\pi]^{d_2}} L_n^{|\bk| + 2\mu +d_1+d_2-1} \big(\xi(\xb, \yb, \ub, \bm v; s, \bm \t) \big) \\
  & \quad \times \prod_{i=1}^{d_2-1} j_{\k_i-\f12}\big ( \sqrt{y_i v_i} \sin \t_i\big) 
   j_{\mu-\f12} \Bigg(\sqrt{\tfrac{\la \xb,\ub\ra+|\yb|\cdot |\vb| + s \sqrt{|\yb|^2-\|\xb\|^2}\sqrt{|\vb|^2-\|\ub\|^2} }{2} }   \sin \t_d \Bigg) \\
 &\quad \times \exp \left( -\sum_{i=1}^{d-1} \sqrt{y_i v_i} \cos \t_i- \sqrt{\tfrac{\la \xb,\ub\ra+|\yb|\cdot |\vb| 
   + s \sqrt{|\yb|^2-\|\xb\|^2}\sqrt{|\vb|^2-\|\ub\|^2} }{2} } \cos \t_{d}\right)\\
  & \quad \times \prod_{i=1}^{d-1} (\sin \t_i)^{2 \k_i} (\sin \t_d)^{2\mu-1} (1-s^2)^{\mu-1} \d \bm{\t} \d s,
\end{align*}
where $\cb_{\mu,\bk} =  c_{\mu-\f12} \cb_{\bk, \mu-\f12}^L$, and
\begin{align*} %\label{eq:eta()BSig}
 \xi(x, \bm y, u, \bm v; \bm t)  = & |\yb'| + |\yb| +|\vb'| + |\vb| + 2 \sum_{i=1}^{d-1} \sqrt{y_i v_i} \cos \t_i \\
   & + 2 \big(\la \xb,\ub\ra+|\yb|\cdot |\vb| + s \sqrt{|\yb|^2-\|\xb\|^2}\sqrt{|\vb|^2-\|\ub\|^2} \big) \cos \t_d.  \notag
\end{align*}
\end{thm}

This addition formula follows immediately from combining the two previous propositions, in which we have used
$j_\a(0) =  2^{-\a} /\Gamma(\a+1)$, as well as 
$$
 (\zeta (\bm x,\bm y,  \bm u,\bm v; s) + |\yb|)|\vb| = \la \xb,\ub\ra+ |\yb|\cdot |\vb|
   + s \sqrt{|\yb|^2-\|\xb\|^2}\sqrt{|\vb|^2-\|\ub\|^2}.
$$

\section{Spherical-Laguerre polynomials on the surface $\SS^{d_1-1}\!\rtimes \RR_\Sigma^{d_2}$}
\setcounter{equation}{0}

We consider orthogonal polynomials on the surface of $\BB^{d_1} \! \rtimes \RR_\Sigma^{d_2}$ in this section,
which is the wrapped product of the unit sphere with $\RR_\Sigma^{d_2}$ defined by 
%Let $d_1$ and $d_2$ be two positive integers and $d_1 \ge 2$. We consider the quadratic surface 
%$\SS^{d_1-1}\!\rtimes \RR_\Sigma^{d_2}$ of the domain $\BB^{d_1}\!\rtimes \RR_\Sigma^{d_2}$, defined by 
$$
  \SS^{d_1-1}\!\rtimes \RR_\Sigma^{d_2} = \left\{(\bm x,\bm y) \in \RR^{d_1} \times \RR_+^{d_2}: \|\bm x\| = |\bm y|\right\}.
$$
For $\bm\kappa \in \RR^{d_2}$ such that $\k_i > -1$, $1 \le i \le d_2-1$, and $\k_{d_2} > - d_1$, we define  
$$
  {\sf W}_{\bk}^{\SS\rtimes \Sigma}(\bm y) =  \prod_{i=1}^{d_2-1} |y_i|^{\k_i} |\bm y|^{\k_{d_2}} \e^{-|\yb'|- |\yb|}, \qquad \yb \in \RR_\Sigma^{d_2},
$$
and consider orthogonal polynomials with respect to the inner product defined by 
\begin{equation} \label{eq:ipdS-Sig}
  \la f, g\ra_{\bk}^{\SS\rtimes \Sigma}  = \sb_{\bk}^{\SS\rtimes \Sigma} 
  \int_{ \SS^{d_1-1}\!\rtimes \RR_\Sigma^{d_2}} f(\bm x,\bm y) g(\bm x,\bm y) \sW_{\bk}^{\SS\rtimes \Sigma}(\bm y) 
     \d \s(\bm x, \bm y), 
\end{equation}
where $\d \s(\bm x, \bm y)$ is the surface measure of $\SS^{d_1-1}\!\rtimes \RR_\Sigma^{d_2}$ and
the constant $\sb_{\bk}^{\SS\rtimes \Sigma}$ satisfies
\begin{equation} \label{eq:b-S-Sig0}
  \sb_{\bk}^{\SS\rtimes \Sigma} = \frac{1}{\o_{d_1}} \bm b^L_{(\bk', \k_{d_2} +d_1 -1)}= \frac{1}{\o_{d_1}}
     \frac{1}{\prod_{i=1}^{d_2-1} \Gamma(\k_i +1) \Gamma(\k_{d_2} + d_1)},
\end{equation}
where $\o_d$ is the surface area of $\sph$. Let 
$\CV_n\big(\SS^{d_1-1}\!\rtimes \RR_\Sigma^{d_2},  {\sf W}_{\bk}^{\SS\rtimes \Sigma}\big)$ be
the space of orthogonal polynomials of degree $n$ with respect to the inner product 
$\la \cdot, \cdot\ra_{\g, \bk}^{\SS\rtimes \Sigma}$. 
 
 We let $\{\hat \Lb_{\bm j, m}^{\bg}: |\bm j| = m, \, \bm j\in \NN_0^{d_2}\}$ be an orthogonal basis for 
$\CV_{m}\left(\RR_\Sigma^{d_2},\bm W^\Sigma_{\bg}\right)$, where $\bg \in \RR^{d_2}$, such as the one 
given in \eqref{eq:basis_S}. Let $\{ \sY_{\ell}^m: 1 \le \ell \le \dim \CH_m^{d_1} \}$ be an orthonormal basis of 
$\CH_n^{d_1}$. 
 
\begin{prop}\label{prop:OP_SS-SIg}
For $\bk \in \RR^{d_2}$ with $\k_i > -1$, $1 \le i \le d_2-1$, and $\k_{d_2} > - d_1$. Let $\a = \k_{d_2} + d_1 -1$. Define  
\begin{equation*}% \label{eq:QbasisV0}
 \sQ_{\bm j, \ell, m}^n(\bm x,\bm y) = \hat \Lb_{\bm j, n-m}^{(\bk, \alpha+2m)}(\bm y) |\bm y|^m 
      \sY_{\ell}^m \left(\frac{\bm x}{|\bm y|}\right). 
\end{equation*}
Then $\{\sQ_{\bm j, \ell, m}^n: |\bm j| = n-m, \, 1 \le \ell \le \dim \CH_m^{d_1}, \, 0 \le m  \le n, \bm j \in \NN_0^{d_2}\}$ 
is an orthogonal basis of $\CV_n(\SS^{d_1-1}\!\rtimes \RR_\Sigma^{d_2}, \sW_{\bk}^{\SS\rtimes \Sigma})$. Moreover, 
the squared norm of $\sQ_{\bm j, \ell, m}^n$ 
is given by 
\begin{equation*} %\label{eq:Q_V0Norm}
   \sh_{m,n}^{\bk}  % = \la \sQ_{\bm j, \ell, m}^n, \sQ_{\bm j,\bm k, m}^n \ra_{\mu, \bk} 
       =   \frac{\bm b^L_{(\bk',\alpha)}}{\bm b^L_{(\bk', \alpha + 2m)}} \bm h_{\bm j,m}^{(\bk', \a+ 2m), L}
       =  (\a+1)_{2m}\bm h_{\bm j,m}^{(\bk', \a+ 2m), L}.  
\end{equation*}
\end{prop}

This is \cite[Proposition 7.6.]{X26}. Similar to that of $\SS^{d_1 -1} \! \rtimes \triangle^{d_2}$, the most 
important family turns out to be when $\k_{d_2} = -1$, since it is the only case for which the orthogonal 
polynomials possess a spectral operator \cite{X26}, and for which the addition formula holds, and only 
in this case. 

Let $\sP_n\big(\sW_{\bk}^{\SS\rtimes \RR_\Sigma}; \cdot, \cdot\big)$ be the reproducing kernel of the space 
$\CV_n\big(\SS^{d_1-1}\!\rtimes \RR_\Sigma^{d_2},  \sW_{\bk}^{\SS\rtimes \Sigma}\big)$. 

\begin{thm}\label{thm:additionS-Sig}
Let $d_1 \ge 2$ and $\k_1,\ldots, k_{d_2-1} \ge -\f12$ and $\k_{d_2} = -1$. Then 
\begin{align*}%\label{eq:additionV0}
 & \sP_n  \big(\sW_{\bk}^{\SS\rtimes \Sigma},  (\bm x, \bm y), (\bm u, \bm v) \big) 
    =  \sc_{\bk} \int_{[0,\pi]^{d_2}} L_{2n}^{|\bk| + d_1+d_2-2}
    \big( \xi_0(\bm x, \bm y, \bm u, \bm v; \bm \t)\big) \notag \\
  & \times \prod_{i=1}^{d_2-1} j_{\k_i-\f12}\big ( \sqrt{y_i v_i} \sin \t_i\big)
     j_{\frac{d_1-3}{2}}  \Big(\sqrt{\tfrac{ \la \xb,\ub\ra +|\yb|\cdot|\vb|}2} \sin \t_{d_2} \Big) \\
  &\times \exp \left( -\sum_{i=1}^{d_2-1} \sqrt{y_i v_i} \cos \t_i-  \f12 \sqrt{\la \xb, \ub\ra + |\yb|\cdot |\vb|}\cos \t_{d_2} \right)
   \prod_{i=1}^{d_2-1} (\sin \t_i)^{2\k_i} (\sin \t_{d_2})^{d_1-3} \d \bm{\t},
\end{align*}
where $\cb_\bk =  \cb_{\bk', \frac{d_1-3}{2}, \frac{d_1-3}{2}}^L$, and
\begin{align*}%\label{eq:eta()SS-Sig}
 \xi_0(\xb, \bm y, \ub, \bm v; \bm \theta)  = & |\yb'| + |\yb| +|\vb'| + |\vb| + 2 \sum_{i=1}^{d-1} \sqrt{y_i v_i} \cos \t_i
  +  \sqrt{(\la \xb,\ub\ra + |\yb|\cdot |\vb|} \cos \t_{d_2}  . \notag
\end{align*} 
\end{thm}

\begin{proof}
The proof is similar to that of Theorem \ref{thm:additionV0}. From the explicit formulas for the orthogonal basis and the
norms in Proposition \ref{prop:OP_SS-SIg}, and applying the addition formula for the spherical harmonics, it follows 
that 
$$
\sP_n  \big(\sW_{\bk}^{\SS\rtimes \Sigma},  (\bm x, \bm y), (\bm u, \bm v) \big) 
 = \bm P_n^{1,d_2}\left(\bm W_{(\frac{d_1-2}{2},\bk')}^{[-1,1] \rtimes \Sigma} \big(|\bm y| \la \xi,\eta\ra,\bm y\big), (|\bm v|, \bm v)\right).  
$$
Consequently, the addition formula for $\sW_{\bk}^{\SS\rtimes \Sigma}$ follows from Proposition \ref{prop:Pn_BSig_d1=1}.
\end{proof}

%\noindent
%{\bf Conflicts of Interest Statement}: I have no conflicts of interest to disclose.

\iffalse
\noindent
{\bf Statements and Declarations. } \\

\noindent {\bf Funding}: This work was partially supported by Simons Foundation Grant \# 849676. \\

\noindent {\bf Competing Interests}: The author has no relevant financial or non-financial interests to disclose. \\

\noindent {\bf Data Availability Statement} Data sharing not applicable, as the article has no associated data. 

\vskip .3in
\fi

\end{document}